\documentclass[letterpaper,11pt,reqno]{amsart}
\usepackage[utf8]{inputenc}
\usepackage{amsmath}
\usepackage{amssymb}
\usepackage{amsthm}
\usepackage{hyperref}
\usepackage[arrow, matrix, curve]{xy}
\usepackage{color}
\usepackage{enumitem}

\DeclareMathOperator{\GL}{GL}

\DeclareMathOperator{\SL}{SL}
\let\sl\relax
\DeclareMathOperator{\sl}{\mathfrak{sl}}

\DeclareMathOperator{\Sp}{Sp}

\DeclareMathOperator{\SU}{SU}
\DeclareMathOperator{\su}{\mathfrak{su}}

\DeclareMathOperator{\Str}{Str}
\DeclareMathOperator{\Aut}{Aut}

\newcommand{\fraka}{\mathfrak{a}}

\newcommand{\frakg}{\mathfrak{g}}
\newcommand{\frakh}{\mathfrak{h}}
\newcommand{\frakk}{\mathfrak{k}}
\newcommand{\frakl}{\mathfrak{l}}
\newcommand{\frakm}{\mathfrak{m}}
\newcommand{\frakn}{\mathfrak{n}}
\newcommand{\frakp}{\mathfrak{p}}
\newcommand{\frakq}{\mathfrak{q}}

\newcommand{\CC}{\mathbb{C}}

\newcommand{\NN}{\mathbb{N}}

\newcommand{\RR}{\mathbb{R}}
\newcommand{\TT}{\mathbb{T}}

\newcommand{\calB}{\mathcal{B}}

\newcommand{\calD}{\mathcal{D}}
\newcommand{\calE}{\mathcal{E}}

\newcommand{\calH}{\mathcal{H}}
\newcommand{\calJ}{\mathcal{J}}
\newcommand{\calL}{\mathcal{L}}
\newcommand{\calO}{\mathcal{O}}

\newcommand{\calU}{\mathcal{U}}
\newcommand{\calV}{\mathcal{V}}
\newcommand{\calW}{\mathcal{W}}

\renewcommand{\1}{\textbf{1}}

\DeclareMathOperator{\Ind}{Ind}
\DeclareMathOperator{\tr}{tr}
\DeclareMathOperator{\ad}{ad}

\DeclareMathOperator{\End}{End}
\DeclareMathOperator{\Hom}{Hom}

\DeclareMathOperator{\id}{id}

\DeclareMathOperator{\const}{const}
\DeclareMathOperator{\diag}{diag}

\DeclareMathOperator{\supp}{supp}

\newcommand{\HS}{{\textup{HS}}}

\newcommand{\ip}[2]{\langle #1,#2 \rangle}

\theoremstyle{plain}
\newtheorem{theorem}{Theorem}[section]
\newtheorem{proposition}[theorem]{Proposition}
\newtheorem{lemma}[theorem]{Lemma}
\newtheorem{corollary}[theorem]{Corollary}

\newtheorem{thmalph}{Theorem}

\theoremstyle{definition}

\newtheorem{remark}[theorem]{Remark}

\numberwithin{equation}{section}

\title[Hol. discrete series in the gen. Whittaker Plancherel formula]{The holomorphic discrete series contribution to the generalized Whittaker Plancherel formula}

\author{Jan Frahm}
\address{Department of Mathematics, Aarhus University, Ny Munkegade 118, 8000 Aarhus C, Denmark}
\email{frahm@math.au.dk}

\author{Gestur \'{O}lafsson}
\address{Department of Mathematics, Louisiana State University, Baton Rouge, LA 70803, USA}
\email{olafsson@math.lsu.edu}

\author{Bent {\O}rsted}
\address{Department of Mathematics, Aarhus University, Ny Munkegade 118, 8000 Aarhus C, Denmark}
\email{orsted@math.au.dk}

\begin{document}

\subjclass[2010]{Primary 22E46; Secondary 43A85.}

\keywords{}
\thanks{Part of the research in this paper was carried out within the online research community on Representation Theory and Noncommutative Geometry sponsored by the American Institute of Mathematics.
The first author was partially supported by a research grant from the Villum Foundation (Grant No. 00025373).
The second author was partially supported by Simons grant 586106.}
\maketitle

\begin{abstract}
For a Hermitian Lie group $G$ of tube type we find the contribution of the holomorphic discrete series to the Plancherel decomposition of the Whittaker space $L^2(G/N,\psi)$, where $N$ is the unipotent radical of the Siegel parabolic subgroup and $\psi$ is a certain non-degenerate unitary character on $N$. The holomorphic discrete series embeddings are constructed in terms of generalized Whittaker vectors for which we find explicit formulas in the bounded domain realization, the tube domain realization and the $L^2$-model of the holomorphic discrete series. Although $L^2(G/N,\psi)$ does not have finite multiplicities in general, the holomorphic discrete series contribution does.\\
Moreover, we obtain an explicit formula for the formal dimensions of the holomorphic discrete series embeddings, and we interpret the holomorphic discrete series contribution to $L^2(G/N,\psi)$ as boundary values of holomorphic functions on a domain $\Xi$ in a complexification $G_\CC$ of $G$ forming a Hardy type space $\calH_2(\Xi,\psi)$.
\end{abstract}


\section*{Introduction}

One of the fundamental problems in representation theory is the decomposition of induced representations into irreducibles. In the context of real reductive groups, unitarily induced representations always decompose uniquely into a direct integral of irreducible unitary representations. A special case of this, inducing from the trivial representation, is the decomposition of $L^2(G/H)$, where $G$ is real reductive and $H\subseteq G$ a closed unimodular subgroup. This decomposition is also referred to as the \emph{Plancherel formula for $G/H$}, and has been studied intensively in particular for reductive subgroups $H\subseteq G$, e.g. for the group case, reductive symmetric spaces or, more recently, real spherical spaces. Somewhat antipodal to reductive subgroups are unipotent subgroups $N\subseteq G$. In this case, the theory becomes richer if we allow to induce from a generic unitary character $\psi:N\to\TT$ and denote $L^2(G/N,\psi)=\Ind_N^G(\psi)$. For a maximal unipotent subgroup, the Plancherel formula for $L^2(G/N,\psi)$, also called \emph{Whittaker Plancherel formula}, was obtained by Harish-Chandra, Wallach and van den Ban. The Whittaker Plancherel formula is formulated in terms of Whittaker vectors which are distribution vectors equivariant under the action of $N$. Whittaker vectors also play an important role in analytic number theory, where they are used to expand automorphic forms into generalized Fourier series. In particular, Whittaker vectors on holomorphic discrete series representations are important in the study of holomorphic automorphic forms on Hermitian groups.

\subsection*{Explicit generalized Whittaker vectors}

In this paper, we study \emph{generalized Whittaker vectors} on holomorphic discrete series, i.e. Whittaker vectors for non-maximal unipotent subgroups $N$. More precisely, we focus on the case where $G$ is a Hermitian Lie group of tube type and $N$ the unipotent radical of the Siegel parabolic subgroup $P=MAN$. In this case, $N$ is abelian, and we choose a generic character $\psi$ on $N$ which is contained in a certain symmetric cone $\Omega$ in the unitary dual of $N$. For a holomorphic discrete series representation $(\sigma,\calH_\sigma)$ of $G$, a \emph{generalized Whittaker vector} for the pair $(N,\psi)$ is a distribution vector $W\in\calH_\sigma^{-\infty}$ such that
$$ W(\sigma(n)v) = \psi(n)W(v) \qquad \mbox{for all }v\in\calH_\sigma^\infty,n\in N. $$

Our first main result is an explicit construction and classification of generalized Whittaker vectors for all holomorphic discrete series representations in three different realizations (bounded domain model, tube domain model and $L^2$-model).

\begin{thmalph}[see Propositions~\ref{prop:WhittakerVectorsL2}, \ref{prop:WhittakerVectorsTOmega} and \ref{prop:WhittakerVectorsD}]\label{thm:MainThmIntroA}
	\begin{enumerate}
		\item The dimension of the space of generalized Whittaker vectors on a holomorphic discrete series representation is equal to the dimension of its lowest $K$-type.
		\item All generalized Whittaker vectors on holomorphic discrete series representations are given by the following explicit formulas: \eqref{eq:WhittakerNL2} and \eqref{eq:WhittakerNbarL2} for the $L^2$-model, \eqref{eq:WhittakerTOmega} for the tube domain realization, and \eqref{eq:WhittakerND} and \eqref{eq:WhittakerNbarD} for the bounded domain realization.
	\end{enumerate}
\end{thmalph}

Particularly interesting are the generalized Whittaker vectors with respect to the unipotent subgroup $\overline{N}$ opposite to $N$ in the $L^2$-model, which are essentially given by a vector-valued multivariable $J$-Bessel function on a symmetric cone $\Omega$ in $\overline{N}$. The study of scalar-valued multivariable Bessel functions goes back to the work of Herz~\cite{Her55} for the case where $\Omega$ is the set of positive definite symmetric matrices. It was generalized to arbitrary symmetric cones by Faraut--Koranyi~\cite{FK94} and Dib~\cite{Dib90}, and to the vector-valued case by Ding--Gross~\cite{DG93}. More recently, the scalar-valued $J$-Bessel function on symmetric matrices has been used in the context of weak Maass forms and weakly holomorphic modular forms for $G=\Sp(n,\RR)$, see e.g. Bruinier--Funke--Kudla~\cite{BFK18} and references therein. We believe that our vector-valued generalizations could help to extend their theory to more general general groups and representations.

We remark that the dimension of the space of generalized Whittaker vectors on holomorphic discrete series representations was already determined by Wallach~\cite{Wal03}. However, he did not obtain explicit formulas for the Whittaker vectors.

\subsection*{Discrete spectrum in the generalized Whittaker Plancherel formula}

For every holomorphic discrete series representation $(\sigma,\calH_\sigma)$, matrix coefficients with respect to generalized Whittaker vectors construct embeddings of the smooth vectors $\calH_\sigma^\infty$ into $C^\infty(G/N,\psi)$, the smooth sections of the line bundle induced by $\psi$. More precisely, if $\calH_\sigma^{-\infty,\psi}$ denotes the space of Whittaker vectors on $\sigma$ for the pair $(N,\psi)$, we have a $G$-equivariant embedding
\begin{equation}
	\calH_\sigma^\infty\otimes\calH_\sigma^{-\infty,\psi}\hookrightarrow C^\infty(G/N,\psi), \quad f\otimes W\mapsto\big[g\mapsto W(\sigma(g)^{-1}f)\big].\label{eq:IntroGenMatrixCoeff}
\end{equation}
We use our explicit formulas for Whittaker vectors to show that all such embeddings actually map into $L^2(G/N,\psi)$, thus constructing discrete components in the Plancherel formula for $L^2(G/N,\psi)$.

\begin{thmalph}[see Theorem~\ref{thm:HolDSContribution}, Theorem~\ref{thm:FormalDimension} and Corollary~\ref{cor:UnitHWContribution}]\label{thm:MainThmIntroB}
	For every holomorphic discrete series representation $(\sigma,\calH_\sigma)$ of $G$, the map \eqref{eq:IntroGenMatrixCoeff} extends to a $G$-equivariant embedding
	$$ \calH_\sigma\otimes\calH_\sigma^{-\infty,\psi}\hookrightarrow L^2(G/N,\psi), $$
	constructing the complete contribution of $\sigma$ to the Plancherel formula for $L^2(G/N,\psi)$. Moreover, there exists an explicit constant $d(\sigma)>0$, the \emph{formal dimension}, and an explicit inner product $\langle\cdot,\cdot\rangle_{\calH_\sigma^{-\infty,\psi}}$ on $\calH_\sigma^{-\infty,\psi}$ (see \eqref{eq:DefInnerProductOnWhittakerVectors} and \eqref{eq:FormulaFormalDimension} for the explicit formulas) such that
	$$ \int_{G/N}W\big(\sigma(g)^{-1}f\big)\overline{W'\big(\sigma(g)^{-1}f'\big)}\,d\dot{g} = \frac{1}{d(\sigma)}\langle f,f'\rangle_{\calH_\sigma}\langle W,W'\rangle_{\calH_\sigma^{-\infty,\psi}} $$
	for all $f,f'\in\calH_\sigma$ and $W,W'\in\calH_\sigma^{-\infty,\psi}$.
\end{thmalph}

We remark that, unless $G$ is locally isomorphic to $\SL(2,\RR)$, the decomposition of $L^2(G/N,\psi)$ does not have finite multiplicities for all representations occurring. However, Theorem~\ref{thm:MainThmIntroB} implies that all holomorphic discrete series representations occur with finite multiplicity inside $L^2(G/N,\psi)$. This is in line with recent results of Aizenbud--Gourevitch~\cite{AG21} who introduced the concept of \emph{$\calO$-spherical homogeneous spaces} for nilpotent coadjoint orbits $\calO$. Using \cite[Theorem B]{AG21}, it can be shown that $G/N$ is indeed $\calO_P$-spherical, where $\calO_P$ is the Richardson orbit associated with the maximal parabolic subgroup $P$. Then, \cite[Theorem D]{AG21} implies that all representations with associated variety equal to $\overline{\calO_P}$, such as the holomorphic discrete series, occur with at most finite multiplicity in $L^2(G/N)$. It is likely that a variant of their result also holds for $L^2(G/N,\psi)$.

We further remark that the representation $L^2(G/N)$ is \emph{tempered} in the sense that the support of the Plancherel measure is contained in the set of tempered representations. This follows from recent results of Benoist--Kobayashi~\cite[Theorem 2.9 and Corollary 3.3]{BK17}. By \cite[Lemma 3.2]{BK17}, this also implies that $L^2(G/N,\psi)$ is tempered, so it is natural to look for discrete series in $L^2(G/N,\psi)$.

\subsection*{Lowest $K$-types in the Whittaker space}

The problem of constructing irreducible subrepresentations of $L^2(G/N,\psi)$ was studied earlier by among others Yamashita (see \cite{Y88} and references therein) from a different point of view under the rubric of generalized Gelfand--Graev representations. He obtained the first part of our Theorem~\ref{thm:MainThmIntroB} solving differential equations on $G/N$ and finding highest weight vectors corresponding to holomorphic discrete series representations. We find new formulas for these highest weight vectors on $G/N$ by different methods. More precisely, we show that every holomorphic discrete series embedding $T:\calH_\sigma\hookrightarrow L^2(G/N,\psi)$ can be described in terms of the lowest $K$-type $(\pi,\calV_\pi)$ of $(\sigma,\calH_\sigma)$, the character $\psi$, and the unique decomposition of an element $g\in G$ into a product
$$ g = p^+(g)k_\CC(g)n_\CC(g), $$
where $k_\CC(g)\in K_\CC$, a complexification of the maximal compact subgroup $K$ of $G$, $p^+(g)\in P^+=\exp\frakp^+$, where $\frakp^+$ is one of the two irreducible components of the Cartan complement $\frakp_\CC$ of $\frakk_\CC$, and $n_\CC(g)\in N_\CC$, a complexification of $N$.

\begin{thmalph}[see Theorem~\ref{thm:LKTonG/N}]\label{thm:MainThmIntroC}
	Every lowest $K$-type $(\pi,\calV_\pi)$ of a holomorphic discrete series representation occurring inside $L^2(G/N,\psi)$ is given by
	\begin{multline*}
		T_{\pi,\eta}:\calV_\pi \hookrightarrow L^2(G/N,\psi),\\
		T_{\pi,\eta}\xi(g) = \psi(n_\CC(g)^{-1})\eta\big(\pi(k_\CC(g)^{-1})\xi\big) \qquad (\xi\in\calV_\pi,g\in G)
	\end{multline*}
	for some $\eta\in\calV_\pi^\vee$.
\end{thmalph}

This formula is in the spirit of \cite{HOO91,OO88,OO91}, where an anologous decomposition $G\subseteq P^+K_\CC H_\CC$ was derived for an affine symmetric space $G/H$ and similar formulas for the lowest $K$-types of holomorphic discrete series in $L^2(G/H)$ were obtained.

\subsection*{The Hardy space}

In the context of affine symmetric spaces, the unitary highest weight contribution to $L^2(G/H)$ can be described in terms of a natural \emph{Hardy space} $\calH_2(\Xi_H)$ on $G/H$. This is a space of holomorphic functions on a domain $\Xi_H$ in a complexification $G_\CC/H_\CC$ of $G/H$, containing $G/H$ in the boundary, together with a natural boundary value map $\beta:\calH_2(\Xi_H)\to L^2(G/H)$ which is an isometry onto the unitary highest weight contribution. One of the key players in this theory is the contraction semigroup $S\subseteq G_\CC$ which contains $G$ in its boundary and to which all unitary highest weight representations extend as holomorphic $^*$-representations.

In our setting, there is an analogous Hardy space $\calH_2(\Xi,\psi)$ of holomorphic functions on a domain $\Xi\subseteq G_\CC$, transforming by $\psi$ under the right action of $N$, and a corresponding boundary value map $\beta:\calH_2(\Xi,\psi)\to L^2(G/N,\psi)$. The semigroup $S$ and in particular the group $G$ act on $\calH_2(\Xi,\psi)$ by the left-regular representation.

\begin{thmalph}[see Theorem~\ref{thm:HardyDec}]\label{thm:MainThmIntroD}
	\begin{enumerate}
		\item The Hardy space $\calH_2(\Xi,\psi)$ decomposes discretely into the direct sum of holomorphic discrete series representations, each occurring with multiplicity equal to the dimension of its lowest $K$-type.
		\item The boundary value map $\beta:\calH_2(\Xi,\psi)\to L^2(G/N,\psi)$ is a $G$-equivariant isometry onto the holomorphic discrete series contribution to the Plancherel formula for $L^2(G/N,\psi)$.
	\end{enumerate}
\end{thmalph}

In Appendix~\ref{App:Hardy} we recall the main ideas about realizing representations in Hardy spaces and their boundary values for the case of the group $G$ and for affine symmetric spaces $G/H$. It is remarkable that the case of $G/N$ can be treated more or less in parallel, one of the main similarities being the role of the closed semigroup $S$ in the complexification of $G$, which consists of all contractions of $G/K$, viewed in its Harish-Chandra realization as a bounded symmetric domain in a complex flag manifold. Then $S$ acts on all representation spaces as well as on the Hardy space, and the domain $\Xi$ is an orbit of the interior of $S$. We believe that (as we do in Section 4) to see the cases $X = G$, $X=G/H$ and $X=G/N$ in parallel is convenient, making it clear that the methods are quite similar, and probably could be applied in other similar situations.

\subsection*{Perspectives}

In a subsequent paper~\cite{FOO22} we study the case of holomorphic discrete series with one-dimensional lowest $K$-type in more detail. More precisely, we shed new light onto the formula for the generating function of the Laguerre functions, previously used by Kostant~\cite[formulas (5.14 \& (5.15)]{Kos00} for $G=\SL(2,\RR)$, in the connection with Whittaker vectors.

Another natural problem in this context is to understand generalized Whittaker vectors on unitary highest weight representations that do not necessarily belong to the holomorphic discrete series. As shown in this paper, these representations do not contribute to $L^2(G/N,\psi)$, but one might still study Whittaker embeddings into $C^\infty(G/N,\psi)$. The same problem can be studied for more singular characters $\psi$ of $N$.

Generalizing even more, one might ask what happens if one removes the assumption on $G$ to be of tube type. In this case, the Shilov boundary $G/P$ of $G/K$ corresponds to a maximal parabolic subgroup $P$ with possibly non-abelian nilradical $N$. In general, the nilradical $N$ is two-step nilpotent, and Yamashita~\cite{Y88} studied the occurrence of holomorphic discrete series representations in $\Ind_N^G(\psi)$, where $\psi$ is an infinite-dimensional irreducible unitary representation of $N$. In a subsequent work~\cite{FOO24}, we generalize some of the methods developed in this paper to the case of non-tube type groups. Recent work by Bruggeman--Miatello~\cite{BM19,BM21} for the group $G=\SU(2,1)$ indicates that such Whittaker models play a role in the study of non-abelian Poincar\'{e} series.

For the Iwasawa decomposition $G = KAN_0$ we may also consider homogeneous vector bundles over $G/N_0$ with fiber a unitary irreducible representation of $N_0$ and analogous imbeddings of holomorphic discrete series; here induction in stages from our $N$ (say tube type) and a character should yield some results, using Mackey theory for $N_0$.

As for the theory of Hardy spaces, we remark that in addition to the Whittaker case $G/N$ studied in this paper and the case of affine symmetric spaces $G/H$ studied in \cite{HOO91} one may consider a version also on $G/K$, since line bundles for a character of $K$ admit embeddings of (finitely many) holomorphic discrete series by the work of Shimeno~\cite{Shi94}. All such Hardy spaces have a reproducing kernel, and one further problem could be to express it as a sum over the holomorphic discrete series in it.

Finally, we believe that our methods in Section~\ref{sec:NKP}, in particular the construction of the kernel function $\Psi(g,z)$, might generalize from holomorphic discrete series to more general discrete series representations, since every discrete series representation of $G$ has a reproducing kernel when realized as kernel of an elliptic differential operator on $G/K$.

\subsection*{Structure of the paper}

In Section~\ref{sec:HolDS} we set up the notation and recall the three different models for the holomorphic discrete series. Explicit formulas for generalized Whittaker vectors are obtained in Section~\ref{sec:GenWhittakerVectors}, thus proving Theorem~\ref{thm:MainThmIntroA}. These formulas are used in Section~\ref{sec:DiscreteSpectrumL2GN} to show Theorem~\ref{thm:MainThmIntroB}. This also leads to a formula for the formal dimension. In Section~\ref{sec:NKP} we first show that $G\subseteq P^+K_\CC N_\CC$, so every $g\in G$ can be decomposed as $g=p^+(g)k_\CC(g)n_\CC(g)$, and then describe the integral kernel of a holomorphic discrete series embedding, and in particular the embedding of the lowest $K$-types, in terms of the projections $k_\CC(g)$ and $n_\CC(g)$. Finally, in Section~\ref{sec:HardySpace} we define the natural Hardy space $\calH_2(\Xi,\psi)$ and relate it to the holomorphic discrete series contribution of $L^2(G/N,\psi)$.

\section{Holomorphic discrete series representations}\label{sec:HolDS}

In this first section we recall three realizations of holomorphic discrete series representations for Hermitian Lie groups of tube type: the bounded domain realization, the tube domain realization and the $L^2$-model. We use the language of Euclidean Jordan algebras and our notation is a mixture of \cite{DG93} and \cite{FK94}.

\subsection{Simple Jordan algebras}

Let $V$ be a simple Euclidean Jordan algebra of dimension $n=\dim(V)$ with unit element $e$. We write $V^\times\subseteq V$ for the open dense subset of invertible elements. The set of invertible squares
$$ \Omega=\{x^2:x\in V^\times\}\subseteq V $$
is a symmetric cone. We write $T_\Omega=V+i\Omega$ for the corresponding tube domain inside the complexification $V_\CC$ of $V$. Note that $V_\CC$ is a simple complex Jordan algebra.

We write $\tr(z)$ for the Jordan trace of $z\in V_\CC$ and note that it is real for $z\in V$ and strictly positive for $z\in\Omega$. The integer $r=\tr(e)$ is called the \emph{rank of $V$}. The trace form on $V_\CC$ is the bilinear form
$$ (z|w)=\tr(z\cdot w) \qquad (z,w\in V_\CC). $$

Further, let $\Delta(z)$ denote the Jordan determinant of $z\in V_\CC$ which is also real for $z\in V$ and positive for $z\in\Omega$. The polynomial $\Delta$ is of degree $r$ and $\Delta(e)=1$.

For $z\in V_\CC$ we write $L(z):V_\CC\to V_\CC$ for the multiplication by $z$. The quadratic representation of $V_\CC$ is the polynomial map
$$ P:V_\CC\to\End(V_\CC), \quad P(z)=2L(z)^2-L(z^2), $$
and it restricts to a map $V\to\End(V)$. We note that for $z\in V_\CC^\times$, the invertible elements in $V_\CC$, the linear map $P(z)$ is invertible. We further define for $z,w\in V_\CC$ the box operator
$$ z\Box w = L(zw)+[L(z),L(w)]\in\End(V_\CC). $$

\subsection{Associated Lie groups and Lie algebras}\label{sec:AssociatedGroups}

The group $G=\Aut(T_\Omega)$ of biholomorphic automorphisms of $T_\Omega$ is a simple Lie group with trivial center. It acts transitively on $T_\Omega$ and the stabilizer $K$ of the element $ie\in T_\Omega$ is a maximal compact subgroup of $G$, hence $T_\Omega\simeq G/K$ is a Riemannian symmetric space. Denote by $\theta$ the corresponding Cartan involution of $G$ and, abusing notation, also on the Lie algebra $\frakg$ of $G$. We write
$$ \frakg = \frakk\oplus\frakp $$
for the decomposition into $\pm1$ eigenspaces of $\theta$.

Let us define some subgroups of $G$. The subgroup
$$ N=\{n_x:x\in V\} $$
of translations $n_x:T_\Omega\to T_\Omega,\ z\mapsto z+x$ is unipotent and occurs as the unipotent radical of a maximal parabolic subgroup $P=LN$ of $G$. The Levi factor can be chosen as the automorphism group of the cone $\Omega$
$$ L=\Aut(\Omega)=\{g\in\GL(V):g\Omega=\Omega\}. $$
It acts naturally on $T_\Omega$ by linear transformations. The group $G$ is generated by $N$, $L$ and the inversion
$$ j:T_\Omega\to T_\Omega,\,z\mapsto-z^{-1}. $$
The parabolic subgroup opposite to $P$ is given by $\overline{P}=jPj^{-1}=L\overline{N}$ with
$$ \overline{N}=\{\overline{n}_v=jn_vj^{-1}:v\in V\}. $$

On the Lie algebra level, we obtain the Gelfand--Naimark decomposition
\begin{equation}
	\frakg=\frakn\oplus\frakl\oplus\overline{\frakn},\label{eq:GelfandNaimark}
\end{equation}
where $\frakl$ is the Lie algebra of $L$ and $\frakn$ and $\overline{\frakn}$ are the Lie algebras of $N$ and $\overline{N}$ given by
$$ \frakn=\left\{N_u=\left.\frac{d}{dt}\right|_{t=0}n_{tu}:u\in V\right\}, \qquad \overline{\frakn} = \left\{\overline{N}_u=\left.\frac{d}{dt}\right|_{t=0}\overline{n}_{tv}:u\in V\right\}. $$
The Lie algebra $\frakl$ of $L$ is spanned by $u\Box v$, $u,v\in V$. We choose a Jordan frame $\{e_1,\ldots,e_r\}$ for $V$, i.e. a maximal set of pairwise orthogonal primitive idempotents such that $e_1+\cdots+e_r=e$. Denoting by $h_j=2L(e_j)\in\frakl$ the multiplication with $2e_j$, we obtain a maximal abelian subspace
$$ \fraka=\bigoplus_{j=1}^r\RR h_j $$
of $\frakl\cap\frakp$ which is also maximal abelian in $\frakp$. The root system $\Sigma(\frakg,\fraka)$ is of type $C_r$ and we write it as
$$ \left\{\frac{1}{2}(\pm\gamma_j\pm\gamma_k):1\leq j<k\leq r\right\}\cup\{\pm\gamma_j:1\leq j\leq r\} $$
with $\gamma_j(h_k)=2\delta_{jk}$. Then the Lie algebra $\frakn$ consists of the root spaces for $\gamma_j$ ($1\leq j\leq r$) and $\frac{1}{2}(\gamma_j+\gamma_k)$ ($1\leq j<k\leq r$). The root spaces for $\pm\gamma_j$ are one-dimensional and the root spaces for $\frac{1}{2}(\pm\gamma_j\pm\gamma_k)$ have a common dimension $d$. Note that $n=r+r(r-1)\frac{d}{2}$. We choose the positive system $\Sigma^+(\frakg,\fraka)$ induced by the ordering $\gamma_1>\gamma_2>\ldots>\gamma_r>0$. The elements $h_j\in\fraka$ and $e_j\in\frakg_{\gamma_j}$ can be completed to an $\sl(2)$-triple $(h_j,e_j,f_j)$ by $f_j\in\frakg_{-\gamma_j}$. Note that $\{\gamma_1,\ldots,\gamma_r\}$ forms a maximal set of strongly orthogonal positive roots.

The complexified Lie algebras $\frakl_\CC$ and $\frakk_\CC$ are conjugate within the complexification $\frakg_\CC$ of $\frakg$. In particular, they are isomorphic, the isomorphism being induced from the homomorphism
$$ \frakk\to\frakl_\CC, \quad N_u+T+\overline{N}_u \mapsto T+2iL(u) \qquad (T\in\frakm\cap\frakk,u\in V). $$
This homomorphism lifts to a Lie group homomorphism 
\begin{equation}
	K\to L_\CC, \quad \,k\mapsto k_c,\label{eq:HomKLC}
\end{equation}
into the complexification
$$ L_\CC=\Str(V_\CC)=\{g\in\GL(V_\CC):P(gz)=gP(z)g'\mbox{ for all }z\in V_\CC\}, $$
the structure group of the complex Jordan algebra $V_\CC$, where $g'\in\GL(V_\CC)$ denotes the adjoint of $g\in\GL(V_\CC)$ with respect to the trace form $(\cdot|\cdot)$ on $V_\CC$. The image of this map is a compact real form of $L_\CC$.

In analogy to the Gelfand-Naimark decomposition \eqref{eq:GelfandNaimark}, we can decompose the complexification of $\frakg$ as
$$ \frakg_\CC = \frakp^+\oplus\frakk_\CC\oplus\frakp^-, $$
where
\begin{equation}
	\frakp^\pm = \{\overline{N}_u\pm2iL(u)-N_u:u\in V_\CC\},\label{eq:p-plus-minus}
\end{equation}
are the two irreducible components of $\frakp_\CC$ as $\frakk_\CC$-module.

\subsection{Unitary highest weight representations}\label{sec:UnitaryHighestWeightReps}

An irreducible unitary representation $(\sigma,\calH_\sigma)$ is called \emph{highest weight representation} if the space $\calH_\sigma^{\frakp^+}$ of vectors annihilated by the action of $\frakp^+$ is non-trivial. In this case, $\calH_\sigma^{\frakp^+}$ is an irreducible finite-dimensional representation of $K$ called the \emph{highest weight space} and the representation $\calH_\sigma$ is uniquely determined by its highest weight space. Using the isomorphism $\frakk_\CC\simeq\frakl_\CC$, we can view the highest weight space as a representation of $\frakl_\CC$. We will consider unitary highest weight representations for which the highest weight space lifts to a representation $(\pi,\calV_\pi)$ of $L_\CC$. In this case, we write $\pi_c$ for the corresponding representation of $K$ on $\calV_\pi$ given by $\pi_c(k)=\pi(k_c)$ with $k_c\in L_\CC$ as in \eqref{eq:HomKLC}.

For such a representation $(\pi,\calV_\pi)$, consider the subspace $\calV_\pi^{\overline{N}_L}$ fixed by the maximal unipotent subgroup $\overline{N}_L$ of $L$ with Lie algebra $\bigoplus_{j<k}\frakg_{-\frac{1}{2}(\gamma_j-\gamma_k)}$. This space is the lowest restricted weight space with respect to the abelian subalgebra $\fraka_\CC\subseteq\frakl_\CC$ and the choice of positive roots introduced in Section~\ref{sec:AssociatedGroups}. We write
$$ -\frac{1}{2}\sum_{j=1}^rm_j\gamma_j $$
for the corresponding lowest restricted weight and define
$$ \omega(\pi) = m_r. $$

The classification of unitary highest weight representations of $G$ has been achieved by Enright--Howe--Wallach~\cite{EHW83} and Jakobsen~\cite{Jak83} and can roughly be stated as follows. For a fixed irreducible finite-dimensional representation $(\pi,\calV_\pi)$ of $L_\CC$ consider the representations $\pi_\nu=\pi\otimes\det^\nu$ for $\nu\in\RR$, where $\det:L_\CC\to\CC^\times$ denotes the determinant of $g\in L_\CC$ as a linear operator on $V_\CC$. Then there exists a set $\calW_\pi$ referred to as the \emph{Berezin--Wallach set} which is the union of an interval of the form $(-\infty,c]$ and finitely many points in $(c,\infty)$ such that there exists a unitary highest weight representation of the universal cover $\widetilde{G}$ of $G$ with highest weight space isomorphic to $\pi_\nu$ if and only if $\nu\in\calW_\pi$. Moreover, the representations $\pi_\nu$ depend analytically on $\nu$ in the sense that the distribution character is analytic in $\nu\in\calW_\pi\subseteq\CC$.

A unitary highest weight representation that is contained in the discrete series for $G$ is called a \emph{holomorphic discrete series representation}. By \cite[Remark 5.12]{DG93}, this is the case if and only if the highest weight space $(\pi,\calV_\pi)$ satisfies $\omega(\pi)>\frac{2n}{r}-1$. Note that $\omega(\pi_\nu)=\omega(\pi)-\frac{2n}{r}\nu$, so for $\nu\in\calW_\pi$ small enough the representation $\pi_\nu$ is a holomorphic discrete series.

\subsection{The tube domain realization}\label{sec:TubeDomainModel}

Let $(\pi,\calV_\pi)$ be an irreducible representation of $L_\CC$ and choose an inner product $\langle\cdot,\cdot\rangle$ on $\calV_\pi$ such that $\langle\pi(g)v,w\rangle=\langle v,\pi(g^*)w\rangle$ ($v,w\in\calV_\pi$, $g\in L_\CC$), where $g^*=\overline{g}'$. For $z\in V_\CC^\times$ the invertible linear map $P(z)\in\GL(V_\CC)$ is in fact contained in the structure group $L_\CC$, and we put
$$ \pi(z) = \pi(P(z)) \in \End(\calV_\pi). $$
Note that a lowest weight vector $\xi\in\calV_\pi$ satisfies
\begin{equation}
	\pi\Big(\sum_{j=1}^r\lambda_je_j\Big)\xi = \lambda_1^{-m_1}\cdots\lambda_r^{-m_r}\xi.\label{eq:LowestWeightSpaceAction}
\end{equation}

For $\omega(\pi)>\frac{2n}{r}-1$, the representation $\pi$ occurs as the lowest $K$-type of a holomorphic discrete series representation of $G$. It can be realized on the Hilbert space
$$ \calH_\pi^2(T_\Omega)=\left\{F:T_\Omega\to\calV_\pi\mbox{ hol.}:\int_{T_\Omega}\langle\pi(y)^{-1}F(z),F(z)\rangle\,\Delta(y)^{-\frac{2n}{r}}\,dz<\infty\right\}. $$
This Hilbert space has reproducing kernel $K_\pi(z,w)=c_\pi\pi(-i(z-\overline{w}))$ for some constant $c_\pi>0$, i.e.
$$ \langle F(z),\xi\rangle = \langle F, K_z\xi\rangle\qquad (\xi\in\calV_\pi), $$
where $K_z\xi(w)=K(w,z)\xi$.

The group $G$ acts on $\calH_\pi^2(T_\Omega)$ by
$$ T_\pi(g)F(z) = \mu_\pi(g^{-1},z)F(g^{-1}\cdot z), $$
where $\mu_\pi(g,z)=\pi(\mu(g,z))$ with
$$ \mu(g,z) = \left(\frac{\partial(g\cdot z)}{\partial z}\right)^{-1}\in L_\CC. $$
This action can be made explicit for elements in $N$, $L$, $\overline{N}$ and for the inversion $j$:
\begin{align}
	T_\pi(n_u)F(z) &= F(z-u), && u\in V,\notag\\
	T_\pi(g)F(z) &= \pi(g)F(g^{-1}z), && g\in L,\notag\\
	T_\pi(\overline{n}_v)F(z) &= \pi(z)\pi(z^{-1}+v)F((z^{-1}+v)^{-1}), && v\in V,\notag\\
	T_\pi(j)F(z) &= \pi(z)F(-z^{-1}).\label{eq:H2TOmegaActionJ}
\end{align}
Differentiating, we obtain the Lie algebra action
\begin{align}
	dT_\pi(N_u) &= -\partial_u, && u\in V,\label{eq:LieAlgActionTOmega1}\\
	dT_\pi(S) &= -\partial_{Sz}+d\pi(S), && S\in\frakl,\label{eq:LieAlgActionTOmega2}\\
	dT_\pi(\overline{N}_v) &= -\partial_{P(z)v}+2d\pi(z\,\square\,v). && v\in V,\label{eq:LieAlgActionTOmega3}
\end{align}

The representation $(T_\pi,\calH_\pi^2(T_\Omega))$ is a unitary highest weight representation with highest weight space isomorphic to the representation $\pi_c$. 

The lowest $K$-type in $\calH_\pi^2(T_\Omega)$ is spanned by the functions
$$ z\mapsto\pi\left(\frac{z+ie}{2i}\right)\xi \qquad (\xi\in\calV_\pi). $$

\subsection{The bounded domain realization}

The tube domain $T_\Omega\subseteq V_\CC$ is biholomorphically equivalent to a bounded domain $D\subseteq V_\CC$. More precisely, define for $z,w\in V_\CC$ the Bergman operator
$$ B(z,w) = I -2z\square w+P(z)P(w) \in \End(V_\CC), $$
and let $D$ denote the connected component of $0$ in the set $\{w\in V_\CC:B(w,\overline{w})\gg0\}$. Then $D$ is a bounded domain and the Cayley transform $p:T_\Omega\to D,\,z\mapsto(z-ie)(z+ie)^{-1}$ is a biholomorphic equivalence with inverse $c:D\to T_\Omega,\,w\mapsto i(e+w)(e-w)^{-1}$ by \cite[Proposition X.4.2~(iii) and Theorem X.4.3]{FK94}.

Note that $B(z,w)\in L_\CC$ for $z,w\in D$ by \cite[Lemma X.4.4~(ii)]{FK94}. The bounded domain realization of the holomorphic discrete series representation with lowest $K$-type $(\pi,\calV_\pi)$ is on the Hilbert space
$$ \calH_\pi^2(D) = \left\{f:D\to \calV_\pi\mbox{ hol.}:\int_D\langle\pi(B(w,\overline{w}))^{-1}f(w),f(w)\rangle h(w)^{-\frac{2n}{r}}\,dw<\infty\right\}, $$
where
$$ h(w)^{-\frac{2n}{r}} = \det(B(w,\overline{w}))^{-1} \qquad (w\in D). $$

The Cayley transform induces a unitary (up to a scalar) isomorphism
$$ \gamma_\pi:\calH^2_\pi(T_\Omega)\to\calH^2_\pi(D), \quad \gamma_\pi F(w) = \pi(e-w)F(c(w)) $$
with inverse
$$ \gamma_\pi^{-1}:\calH^2_\pi(D)\to\calH^2_\pi(T_\Omega), \quad \gamma_\pi^{-1}f(z) = \pi\left(\frac{z+ie}{2i}\right)f(p(z)). $$
This defines a realization $D_\pi$ of the holomorphic discrete series on $\calH_\pi^2(D)$ by
$$ D_\pi(g) = \gamma_\pi\circ T_\pi(g)\circ\gamma_\pi^{-1} \qquad (g\in G). $$
It is given by
$$ D_\pi(g)f(w) = \nu_\pi(g_c^{-1},w)f(g_c^{-1}w), $$
where $g_c\in\Aut(D)$ is defined by $g_c=p\circ g\circ p^{-1}=c^{-1}\circ g\circ c$ and $\nu_\pi(g_c^{-1},w)=\pi(\nu(g_c^{-1},w))$ with
$$ \nu(g_c^{-1},w) = \left(\frac{\partial(g_c^{-1}\cdot w)}{\partial w}\right)^{-1} \in L_\CC. $$
In particular, $j_c=-\id_D$, so
\begin{equation}
	D_\pi(j)f(w) = \pi(-I)f(-w).\label{eq:DActionJ}
\end{equation}

The $K$-finite vectors in $\calH_\pi^2(D)$ are the $\calV_\pi$-valued polynomials, and the lowest $K$-type is spanned by the constant functions
$$ w\mapsto\xi \qquad (\xi\in\calV_\pi). $$

\subsection{The $L^2$-model}

For an irreducible representation $(\pi,\calV_\pi)$ of $L_\CC$ with $\omega(\pi)>\frac{2n}{r}-1$ the following integral converges absolutely:
\begin{equation}
	\widetilde{\Gamma}_\pi = \int_\Omega e^{-2\tr(u)}\pi(u)^{-1}\Delta(u)^{-\frac{2n}{r}}\,du.\label{eq:DefGammaFunction}
\end{equation}
We have $\widetilde{\Gamma}_\pi \in\End(\calV_\pi)^{L\cap K}$ and $\widetilde{\Gamma}_\pi$ is symmetric and positive definite. In particular, in the special case where $\pi|_{L\cap K}$ is irreducible, $\widetilde{\Gamma}_\pi$ is a scalar multiple of the identity, the scalar being (see \cite[Corollary 2.6 and Section 3.7]{DG93})
$$ \widetilde{\Gamma}_\pi(\alpha)=\const\times\,2^{-m_1-\cdots-m_r}\prod_{j=1}^r\Gamma\left(m_j-\frac{n}{r}-(j-1)\frac{d}{2}\right), $$
where the constant only depends on the dimension $n$ and rank $r$ of $V$.

The Hilbert space for the $L^2$-model of the holomorphic discrete series representation with lowest $K$-type $\pi$ is
$$ L^2_\pi(\Omega) = \left\{f:\Omega\to\calV_\pi:\int_\Omega\langle\widetilde{\Gamma}_\pi\pi(u^{\frac{1}{2}})^{-1} f(u),\pi(u^{\frac{1}{2}})^{-1}f(u)\rangle\,\Delta(u)^{-\frac{n}{r}}\,du<\infty\right\}. $$
with the obvious inner product. The Laplace transform
$$ \calL_\pi:L^2_\pi(\Omega)\to\calH^2_\pi(T_\Omega), \quad \calL_\pi f(z) = (2\pi)^{-\frac{n}{2}}\int_\Omega e^{i(z|u)}\pi(u)^{-1}f(u)\Delta(u)^{-\frac{n}{r}}\,du $$
is an isometric isomorphism (up to scalar) with inverse
$$ \calL_\pi^{-1}F(u) = (2\pi)^{-\frac{n}{2}}\Delta(u)^{\frac{n}{r}}\int_V e^{-i(x+iy|u)}\pi(u)F(x+iy)\,dx \qquad \mbox{for fixed $y\in\Omega$.} $$

Using the Laplace transform $\calL_\pi$, one can realize the holomorphic discrete series on $L^2_\pi(\Omega)$ by
$$ R_\pi(g) = \calL_\pi^{-1}\circ T_\pi(g)\circ\calL_\pi \qquad (g\in G). $$
By \cite[Theorem 5.4]{DG93} we have
\begin{align}
	R_\pi(n_u)f(x) &= e^{-i(x|u)}f(x), && u\in V,\label{eq:L2ActionN}\\
	R_\pi(g)f(x) &= \pi(g^*)^{-1}f(g^*x)&& g\in L,\label{eq:L2ActionL}\\
	R_\pi(j)f(x) &= \int_\Omega \calJ_\pi(u,x)f(u)\Delta(u)^{-\frac{n}{r}}\,du,\label{eq:L2ActionJ}
\end{align}
where $\calJ_\pi(u,v)\in\End(\calV_\pi)$ denotes the operator-valued Bessel function associated with $\pi$ (see \cite[Definition 3.5]{DG93}). Note that we use the notation $\calJ_\pi$ instead of $K_\pi$ as in \cite{DG93}, because for $V=\RR$ it essentially reduces to a classical $J$-Bessel function (see \cite[Chapter XV.2]{FK94}). The function $\calJ_\pi(u,v)$ is uniquely determined by the equivariance property
$$ \calJ_\pi(gu,(g^*)^{-1}v) = \pi(g)\circ\calJ_\pi(u,v)\circ\pi(g^*) \qquad (g\in L) $$
and the Laplace transform of $\calJ_\pi(u)=\calJ_\pi(u,e)$:
$$ \int_\Omega e^{i(z|u)}\pi(u)^{-1}\calJ_\pi(u)\Delta(u)^{-\frac{n}{r}}\,du = e^{-i\tr(z^{-1})}\pi(z) \qquad (z\in T_\Omega). $$

The Lie algebra action $dR_\pi$ associated with the representation $R_\pi$ can be computed by applying the Laplace transform to \eqref{eq:LieAlgActionTOmega1}, \eqref{eq:LieAlgActionTOmega2} and \eqref{eq:LieAlgActionTOmega3}, and it is given by
\begin{align*}
	dR_\pi(N_u) &= -i(u|x), && u\in V,\\
	dR_\pi(S) &= \partial_{S^*x}-d\pi(S^*), && S\in\frakl,\\
	dR_\pi(\overline{N}_v) &= i(v|\calB_\pi), && v\in V,
\end{align*}
where $\calB_\pi:C^\infty(V,\calV_\pi)\to C^\infty(V,\calV_\pi)\otimes V_\CC$ is the vector-valued Bessel operator given by
$$ (v|\calB_\pi f(x)) =  \sum_{\alpha,\beta}(P(\widehat{e}_\alpha,\widehat{e}_\beta)x|v)\partial_{e_\alpha}\partial_{e_\beta}f(x)-2\sum_\alpha d\pi(v\,\Box\,\widehat{e}_\alpha)\partial_{e_\alpha}f(x), $$
$(e_\alpha)$ denoting a basis of $V$ and $(\widehat{e}_\alpha)$ denoting the dual basis with respect to the trace form.

The $K$-finite vectors in $L^2_\pi(\Omega)$ are the products of $\calV_\pi$-valued polynomials on $\Omega$ and the function $e^{-\tr(x)}$, and the lowest $K$-type is spanned by the functions
$$ f_\xi(x)=e^{-\tr(x)}\xi, \qquad (\xi\in\calV_\pi). $$
For later computations, we need to determine the action of $K$ on the lowest $K$-type. Recall that $\frakk_\CC$ is isomorphic to $\frakl_\CC$. More precisely, the isomorphism is induced by the Lie algebra homomorphism $\frakk\to\frakl_\CC$ given by
$$ N_u+T+\overline{N}_u \mapsto T+2iL(u) \qquad (T\in\frakm\cap\frakk,u\in V). $$
This homomorphism lifts to a Lie group homomorphism $K\to L_\CC$ which we denote by $k\mapsto k_c$. The image of this map is a maximal compact subgroup of $L_\CC$.

\begin{lemma}\label{lem:ActionLKT}
	$R_\pi(k)f_\xi=f_{\pi_c(k)\xi}$ for all $k\in K$, $\xi\in\calV_\pi$.
\end{lemma}

\begin{proof}
	It suffices to show that $dR_\pi(N_u+T+\overline{N}_u)f_\xi=f_{d\pi(T+2iL(u))\xi}$. This is easily derived from the formulas above, the key computation being $(u|\calB_\pi)f_\xi=(x|u)f_\xi+2f_{d\pi(L(u))\xi}$.
\end{proof}

\section{Generalized Whittaker vectors on holomorphic discrete series}\label{sec:GenWhittakerVectors}

For $v\in V$ define $\psi(n_u)=\psi_v(n_u)=e^{-i(u|v)}$, then $\psi_v$ is a unitary character of $N$ and every such character is of this form. Similarly, we obtain the unitary characters of $\overline{N}$ by $\overline{\psi}_v(\overline{n}_u)=e^{-i(u|v)}$. (Note that $\overline{\psi}$ does not denote the character of $N$ which is complex conjugate to $\psi$.)

For a unitary representation $(\sigma,\calH_\sigma)$ of $G$ we write $\calH_\sigma^\infty$ for the subspace of smooth vectors and $\calH_\sigma^{-\infty}=(\calH_\sigma^\infty)'$ for its dual space, the space of distribution vectors. The space of generalized Whittaker vectors on $\sigma$ with respect to $\psi$ resp. $\overline{\psi}$ is defined by
$$ \calH_\sigma^{-\infty,\psi} = \{u\in\calH_\sigma^{-\infty}:u\big(\sigma(n)\varphi\big)=\psi(n)u\big(\varphi\big)\mbox{ for all }\varphi\in\calH_\sigma^\infty,n\in N\} $$
resp.
$$ \calH_\sigma^{-\infty,\overline{\psi}} = \{u\in\calH_\sigma^{-\infty}:u\big(\sigma(\overline{n})\varphi\big)=\overline{\psi}(\overline{n})u\big(\varphi\big)\mbox{ for all }\varphi\in\calH_\sigma^\infty,\overline{n}\in\overline{N}\}. $$

In this section we find explicit expressions for all generalized Whittaker vectors on holomorphic discrete series representations in the three realizations described in Section~\ref{sec:HolDS}.

\subsection{The $L^2$-model}

To construct distribution vectors in $L^2_\pi(\Omega)^{-\infty}=(L^2_\pi(\Omega)^\infty)'$, we first give a lower and an upper bound for the space $L^2_\pi(\Omega)^\infty$ of smooth vectors in $L^2_\pi(\Omega)$. This statement is probably well-known, but we could not find it in the literature in this generality, so we include a short proof for convenience.

\begin{lemma}
	The following inclusions hold with continuous embeddings:
	$$ C_c^\infty(\Omega)\otimes\calV_\pi\subseteq L^2_\pi(\Omega)^\infty\subseteq C^\infty(\Omega)\otimes\calV_\pi. $$
\end{lemma}

\begin{proof}
	First note that the Lie algebra representation $dR_\pi$ of $\frakg$ on $L^2_\pi(\Omega)^\infty$ is given by differential operators with polynomial coefficients and can therefore be extended to a representation on $\calD'(\Omega)\otimes\calV_\pi$. Using this extension, we can characterize the space of smooth vectors by
	$$ L^2_\pi(\Omega)^\infty=\{f\in L^2_\pi(\Omega):dR_\pi(D)f\in L^2_\pi(\Omega)\mbox{ for all }D\in\calU(\frakg_\CC)\}, $$
	where $\calU(\frakg_\CC)$ denotes the universal enveloping algebra of $\frakg_\CC$. Since every differential operator with smooth coefficients maps $C_c^\infty(\Omega)\otimes\calV_\pi\subseteq L^2_\pi(\Omega)$ into itself, the first inclusion follows. It is continuous since $C_c^\infty(\Omega)\otimes\calV_\pi$ embeds continuously into $L^2_\pi(\Omega)$ and differential operators act continuously on $C_c^\infty(\Omega)\otimes\calV_\pi$. For the second inclusion it suffices to consider the action of $\frakl$ which is given by $dR_\pi(S)=\partial_{S^*x}-d\pi(S^*)$. Since $d\pi(S^*)$ clearly maps $L^2_\pi(\Omega)$ into itself, the space $L^2_\pi(\Omega)^\infty$ is contained in
	$$ \{f\in L^2_\pi(\Omega):\partial_{S_1^*x}\cdots\partial_{S_n^*x}f\in L^2_\pi(\Omega)\mbox{ for all }S_1,\ldots,S_n\in\frakl,n\in\NN\}. $$
	Note that the group $L$ acts transitively on $\Omega$ and hence the vector fields in $\frakl$ span the tangent space at each point of $\Omega$. Applying the Sobolev Embedding Theorem locally shows that $L^2_\pi(\Omega)^\infty$ embeds continuously into $C^k(\Omega)\otimes\calV_\pi$ for every $k\in\NN$, hence into $C^\infty(\Omega)\otimes\calV_\pi$.
\end{proof}

Dualizing, we find
$$ \calE'(\Omega)\otimes\calV_\pi^\vee\subseteq L^2_\pi(\Omega)^{-\infty}\subseteq\calD'(\Omega)\otimes\calV_\pi^\vee. $$
Hence, for $v\in\Omega$ and $\eta\in\calV_\pi^\vee=\Hom_\CC(\calV_\pi,\CC)$, the Dirac distribution $W_{v,\eta}^\Omega=\delta_v\otimes\eta$ defined by
\begin{equation}
	W_{v,\eta}^\Omega(f) = \eta\big(f(v)\big) \qquad (f\in L^2_\pi(\Omega)^\infty)\label{eq:WhittakerNL2}
\end{equation}
is contained in $L^2_\pi(\Omega)^{-\infty}$. We further denote by $\overline{W}_{v,\eta}^\Omega=\eta\circ\calJ_\pi(\cdot,v)\in\calD'(\Omega)\otimes\calV_\pi^\vee$ the distribution given by
\begin{equation}
	\overline{W}_{v,\eta}^\Omega(f) = \int_\Omega\eta\big(\calJ_\pi(u,v)f(u)\big)\Delta(u)^{-\frac{n}{r}}\,du \qquad (f\in C_c^\infty(\Omega)\otimes\calV_\pi).\label{eq:WhittakerNbarL2}
\end{equation}
A priori it is not clear that $\overline{W}_{v,\eta}^\Omega$ extends to a distribution vector in $L^2_\pi(\Omega)^{-\infty}$. This is part of the following statement which classifies Whittaker vectors in the $L^2$-model of the holomorphic discrete series:

\begin{proposition}\label{prop:WhittakerVectorsL2}
	Let $v\in\Omega$.
	\begin{enumerate}
		\item The space $L^2_\pi(\Omega)^{-\infty,\psi_v}$ is spanned by the distribution vectors $W_{v,\eta}^\Omega$, $\eta\in\calV_\pi^\vee$.
		\item The space $L^2_\pi(\Omega)^{-\infty,\overline{\psi}_v}$ is spanned by the distribution vectors $\overline{W}_{v,\eta}^\Omega$, $\eta\in\calV_\pi^\vee$.
	\end{enumerate}
	In particular,
	$$ \dim L^2_\pi(\Omega)^{-\infty,\psi_v} = \dim L^2_\pi(\Omega)^{-\infty,\overline{\psi}_v} = \dim\pi. $$
\end{proposition}

\begin{proof}
	\begin{enumerate}
		\item For $f\in L^2_\pi(\Omega)^\infty$ and $n=n_u\in N$ we have
		$$ W_{v,\eta}^\Omega\big(R_\pi(n_u)f\big) = \eta\big(R_\pi(n_u)f(v)\big) = \eta\big(e^{-i(u|v)}f(v)\big) = \psi_v(n_u)W_{v,\eta}^\Omega(f), $$
		so $W_{v,\eta}^\Omega\in L^2_\pi(\Omega)^{-\infty,\psi_v}$. Conversely, if $T\in L^2_\pi(\Omega)^{-\infty,\psi_v}\subseteq\calD'(\Omega)\otimes\calV_\pi^\vee$ then
		$$ T\left(\Big(e^{-i(u|\cdot)}-e^{-i(u|v)}\Big)f\right) = T\big(R_\pi(n_u)f-\psi_v(n_u)f\big) = 0 $$
		for all $f\in C_c^\infty(\Omega)\otimes\calV_\pi$ and $u\in V$. It follows that $\supp T\subseteq\{v\}$, so $T$ is a linear combinations of derivatives of the Dirac distribution at $v$. Since derivatives of $x\mapsto e^{-i(u|x)}-e^{-i(u|v)}$ do not vanish at $x=v$ for all $u\in V$, it follows that $T$ is of order $0$ and hence $T=\delta_v\otimes\eta$ for some $\eta\in\calV_\pi^\vee$.
		\item To obtain generalized Whittaker vectors for the opposite unipotent subgroup $\overline{N}=jNj^{-1}$ we use the fact that $j\overline{N}j^{-1}=N$ and $\overline{\psi}_v(\overline{n})=\psi_v(j\overline{n}j^{-1})$. Then the map
		$$ L^2_\pi(\Omega)^{-\infty,\psi_v} \to L^2_\pi(\Omega)^{-\infty,\overline{\psi}_v}, \quad T\mapsto T\circ R_\pi(j) $$
		is an isomorphism. Since the action of $j$ in $R_\pi$ is given by \eqref{eq:L2ActionJ}, the result follows.\qedhere
	\end{enumerate}
\end{proof}

\begin{remark}
	The explicit formulas in Proposition~\ref{prop:WhittakerVectorsL2} for Whittaker vectors in the $L^2$-model were previously obtained by Kostant~\cite[Theorem 5.15]{Kos00} for the case $G=\SL(2,\RR)$, and by the first author~\cite[Theorem 6.11]{Moe13} for general $G$ and holomorphic discrete series with one-dimensional lowest $K$-type. Both references also consider Whittaker vectors for characters $\psi=\psi_v$ with $v\in\overline{\Omega}$ and also study the analytic continuation of the holomorphic discrete series.
\end{remark}

\subsection{The tube domain realization}\label{sec:TubeDomainModel}

We now find the generalized Whittaker vectors in the realization $\calH^2_\pi(T_\Omega)^{-\infty}$ by applying the Laplace transform to the generalized Whittaker vectors in $L^2_\pi(\Omega)^{-\infty}$ obtained in Proposition~\ref{prop:WhittakerVectorsL2}. For this, we first identify $\calH^2_\pi(T_\Omega)^{-\infty}$ with a space of antiholomorphic functions on $T_\Omega$ with values in $\calV_\pi^\vee$.

\begin{lemma}
	Every distribution vector $T\in\calH^2_\pi(T_\Omega)^{-\infty}$ is given by an antiholomorphic function $\Pi:T_\Omega\to \calV_\pi^\vee$ in the sense that
	\begin{equation}
		T(F) = \int_{T_\Omega}\Pi(z)\big(\pi(y)^{-1}F(z)\big)\Delta(y)^{-\frac{2n}{r}}\,dz \qquad (F\in\calH_\pi^2(T_\Omega)^\infty).\label{eq:DistribVectorHolomFct}
	\end{equation}
\end{lemma}

\begin{proof}
	The argument given in \cite[beginning of Section 3]{CF04} generalizes to the vector-valued situation and we briefly sketch it to keep this work self-contained. The distribution vectors are given by $\calH_\pi^2(T_\Omega)^{-\infty}=\overline{dT_\pi(\calU(\frakg_\CC))\calH_\pi^2(T_\Omega)}$ by \cite[Théorème 1.3]{Car76}, where $\calU(\frakg_\CC)$ denotes the universal enveloping algebra of $\frakg_\CC$ and $\overline{V}$ the complex conjugate of a vector space $V$. Since $\calH_\pi^2(T_\Omega)$ consists of holomorphic functions and $dT_\pi$ is given by holomorphic differential operators (see \eqref{eq:LieAlgActionTOmega1}, \eqref{eq:LieAlgActionTOmega2} and \eqref{eq:LieAlgActionTOmega3}), the result follows.
\end{proof}

For $v\in\Omega$ and $\eta\in\calV_\pi^\vee$ we define antiholomorphic functions $W_{v,\eta}^{T_\Omega}$ and $\overline{W}_{v,\eta}^{T_\Omega}$ on $T_\Omega$ with values in $\calV_\pi^\vee$ by
\begin{equation}
	W_{v,\eta}^{T_\Omega}(z) = e^{-i(\overline{z}|v)}\eta \qquad \overline{W}_{v,\eta}^{T_\Omega}(z)=e^{i(\overline{z}^{-1}|v)}\eta\circ\pi(z).\label{eq:WhittakerTOmega}
\end{equation}
It is easy to see that the functions $W_{v,\eta}^{T_\Omega}$ and $\overline{W}_{v,\eta}^{T_\Omega}$ satisfy the right equivariance condition with respect to the action of $N$ and $\overline{N}$, but it is not clear a priori that they define distribution vectors in $\calH_\pi^2(T_\Omega)^{-\infty}$. This is part of the following classification statement:

\begin{proposition}\label{prop:WhittakerVectorsTOmega}
	Let $v\in\Omega$.
	\begin{enumerate}
		\item\label{prop:WhittakerVectorsTOmega1} The space $\calH^2_\pi(T_\Omega)^{-\infty,\psi_v}$ is spanned by the distribution vectors $W_{v,\eta}^{T_\Omega}$, $\eta\in\calV_\pi^\vee$.
		\item\label{prop:WhittakerVectorsTOmega2} The space $\calH^2_\pi(T_\Omega)^{-\infty,\overline{\psi}_v}$ is spanned by the distribution vectors $\overline{W}_{v,\eta}^{T_\Omega}$, $\eta\in\calV_\pi^\vee$.
	\end{enumerate}
\end{proposition}

\begin{proof}
	It suffices to show \eqref{prop:WhittakerVectorsTOmega1}, then \eqref{prop:WhittakerVectorsTOmega2} follows by applying $T_\pi(j)$ (using \eqref{eq:H2TOmegaActionJ} and the substitution $z\mapsto-z^{-1}$) similar as in the proof of Proposition~\ref{prop:WhittakerVectorsL2}. To prove \eqref{prop:WhittakerVectorsTOmega1} we show that $W_{v,\eta}^{T_\Omega}$ is the Laplace transform of a Whittaker vector $W_{v,\eta'}^\Omega\in L^2_\pi(\Omega)^{-\infty}$ for some $\eta'\in\calV_\pi^\vee$. For $f\in L^2_\pi(\Omega)^\infty$ we have
	\begin{align*}
		& W_{v,\eta}^{T_\Omega}\big(\calL_\pi f\big)\\
		 ={}& (2\pi)^{-\frac{n}{2}}\int_{T_\Omega}\int_\Omega 	e^{-i(\overline{z}|v)}e^{i(z|u)}\eta(\pi(y)^{-1}\pi(u)^{-1}f(u))\Delta(u)^{-\frac{n}{r}}\Delta(y)^{-\frac{2n}{r}}\,du\,dz\\
		={}& (2\pi)^{-\frac{n}{2}}\int_{T_\Omega}\int_\Omega 	e^{i(x|u-v)}e^{-(y|u+v)}\eta(\pi(y)^{-1}\pi(u)^{-1}f(u))\Delta(u)^{-\frac{n}{r}}\Delta(y)^{-\frac{2n}{r}}\,du\,dz\\
		={}& (2\pi)^{\frac{n}{2}}\Delta(v)^{-\frac{n}{r}}\int_\Omega e^{-2(y|v)}\eta(\pi(y)^{-1}\pi(v)^{-1}f(v))\Delta(y)^{-\frac{2n}{r}}\,dy\\
		={}& (2\pi)^{\frac{n}{2}}\Delta(v)^{-\frac{n}{r}}\eta\left(\left(\int_\Omega e^{-2(y|v)}\pi(y)^{-1}\Delta(y)^{-\frac{2n}{r}}\,dy\right)\pi(v)^{-1}f(v)\right)\\
		={}& W_{v,\eta'}^\Omega(f),
	\end{align*}
	where
	$$ \eta'=\eta\circ\left((2\pi)^{\frac{n}{2}}\Delta(v)^{-\frac{n}{r}}\int_\Omega e^{-2(y|v)}\pi(y)^{-1}\pi(v)^{-1}\Delta(y)^{-\frac{2n}{r}}\,dy\right) \in \calV_\pi^\vee. $$
	Note that the integral converges absolutely. In fact, for $v=e$ we obtain $\eta'=(2\pi)^{\frac{n}{2}}\eta\circ\widetilde{\Gamma}_\pi$.
\end{proof}

\subsection{The bounded domain realization}\label{sec:BoundedDomainModel}

Finally, we obtain an expression for the Whittaker vectors in the bounded symmetric domain picture, i.e. in $\calH^2_\pi(D)^{-\infty}$. As for $\calH_\pi^2(T_\Omega)^{-\infty}$, we identify $\calH_\pi^2(D)^{-\infty}$ with a space of antiholomorphic functions $\Pi:D\to\calV_\pi^\vee$ by
$$ \Pi(f) = \int_D\Pi(w)\big(\pi(B(w,\overline{w}))^{-1}f(w)\big) h(w)^{-\frac{2n}{r}}\,dw \qquad (f\in\calH_\pi^2(D)^\infty). $$

For $v\in\Omega$ and $\eta\in\calV_\pi^\vee$ we define the following antiholomorphic functions on $D$ with values in $\calV_\pi^\vee$:
\begin{align}
	W^D_{v,\eta}(w) &= W^{T_\Omega}_{v,\eta}(c(w))\circ\pi(e-\overline{w}) = e^{-((e+\overline{w})(e-\overline{w})^{-1}|v)}\cdot\eta\circ\pi(e-\overline{w}),\label{eq:WhittakerND}\\
	\overline{W}^D_{v,\eta}(w) &= \overline{W}^{T_\Omega}_{v,\eta}(c(w))\circ\pi(e-\overline{w})\circ\pi(-I) = e^{-((e-\overline{w})(e+\overline{w})^{-1}|v)}\cdot\eta\circ\pi(e+\overline{w}).\label{eq:WhittakerNbarD}
\end{align}
Since $(e+\overline{w})(e-\overline{w})^{-1}=2(e-\overline{w})^{-1}-e$ and $(e-\overline{w})(e+\overline{w})^{-1}=2(e+\overline{w})^{-1}-e$ we can also write
\begin{align*}
	W^D_{v,\eta}(w) &= \const\times\,e^{-2\tr v(e-\overline{w})^{-1}}\cdot\eta\circ\pi(e-\overline{w}),\\
	\overline{W}^D_{v,\eta}(w) &= \const\times\,e^{-2\tr v(e+\overline{w})^{-1}}\cdot\eta\circ\pi(e+\overline{w}).
\end{align*}

\begin{proposition}\label{prop:WhittakerVectorsD}
	Let $v\in\Omega$.
	\begin{enumerate}
		\item\label{prop:WhittakerVectorsD1} The space $\calH^2_\pi(D)^{-\infty,\psi_v}$ is spanned by the distribution vectors $W^D_{v,\eta}$, $\eta\in\calV_\pi^\vee$.
		\item\label{prop:WhittakerVectorsD2} The space $\calH^2_\pi(D)^{-\infty,\overline{\psi}_v}$ is spanned by the distribution vectors $\overline{W}^D_{v,\eta}$, $\eta\in\calV_\pi^\vee$.
	\end{enumerate}
\end{proposition}

\begin{proof}
	Since the Cayley transform $\gamma_\pi:\calH_\pi^2(T_\Omega)\to\calH_\pi^2(D)$ is a unitary (up to a scalar $c\neq0$) isomorphism, we have
	\begin{multline*}
		W^D_{v,\eta}\big(\gamma_\pi F\big) = \int_D e^{-((e+\overline{w})(e-\overline{w})^{-1}|v)}\eta\big(\pi(e-\overline{w})\pi(B(w,\overline{w}))^{-1}\gamma_\pi F(w)\big) h(w)^{-\frac{2n}{r}}\,dw\\
		= c\int_{T_\Omega} e^{-i(\overline{z}|v)}\eta\big(\pi(y)^{-1}F(z)\big)\Delta(y)^{-\frac{2n}{r}}\,dz = c\,W^{T_\Omega}_{v,\eta}(F).
	\end{multline*}
	This shows \eqref{prop:WhittakerVectorsD1}, and \eqref{prop:WhittakerVectorsD2} follows by applying $D_\pi(j)$ (see \eqref{eq:DActionJ}).
\end{proof}

\begin{remark}
	In the case of holomorphic discrete series with one-dimensional lowest $K$-type one can use the characterization of distribution vectors in $\calH_\pi^2(D)^{-\infty}$ due to Ch\'{e}bli--Faraut~\cite{CF04} to show that the antiholomorphic functions above are indeed distribution vectors and hence Whittaker vectors. With a little more effort, solving certain differential equations on $D$, it is also possible to give an independent proof of their exhaustion, hence a more direct proof of Proposition~\ref{prop:WhittakerVectorsD}. However, we remark that the results in \cite{CF04} are not available for higher-dimensional lowest $K$-types to the best of our knowledge.
\end{remark}

\section{Discrete spectrum in $L^2(G/N,\psi)$}\label{sec:DiscreteSpectrumL2GN}

For a unitary character $\psi$ of $N$ let
$$ L^2(G/N,\psi) = \left\{f:G\to\CC:f(gn)=\psi(n)^{-1}f(g),\int_{G/N}|f(g)|^2\,d(gN)<\infty\right\}. $$
We use the explicit formulas for Whittaker vectors in the $L^2$-model to obtain the holomorphic discrete series contribution to the Plancherel formula for $L^2(G/N,\psi)$ in the case where $\psi=\psi_v$, $v\in\Omega$.

\subsection{The holomorphic discrete series contribution to $L^2(G/N,\psi)$}

Recall the holomorphic discrete series $(R_\pi,L^2_\pi(\Omega))$ for an irreducible representation $(\pi,\calV_\pi)$ of $L_\CC$ with $\omega(\pi)>\frac{2n}{r}-1$. Note that by Proposition~\ref{prop:WhittakerVectorsL2}, the space $L^2_\pi(\Omega)^{-\infty,\psi_v}$ of Whittaker vectors on $\pi$ is finite-dimensional. We obtain a family of embeddings of the smooth vectors $L^2_\pi(\Omega)^\infty$ into
$$ C^\infty(G/N,\psi_v) = \{f:G\to\CC\mbox{ smooth}:f(gn)=\psi_v(n)^{-1}f(g)\} $$
in terms of matrix coefficients:
\begin{equation}
	\iota_\pi:L^2_\pi(\Omega)^\infty\otimes L^2_\pi(\Omega)^{-\infty,\psi_v}\to C^\infty(G/N,\psi_v), \quad \iota_\pi(f\otimes W)(g) = W\big(R_\pi(g)^{-1}f\big).\label{eq:DefIotaPi}
\end{equation}
Clearly, $\iota_\pi$ is $G$-equivariant with respect to the action of $G$ on $L^2_\pi(\Omega)^\infty\otimes L^2_\pi(\Omega)^{-\infty,\psi_v}$ by $R_\pi$ on the first factor and the trivial representation on the second factor.

\begin{theorem}\label{thm:HolDSContribution}
	For every irreducible representation $(\pi,\calV_\pi)$ of $L_\CC$ with $\omega(\pi)>\frac{2n}{r}-1$, the map $\iota_\pi$ extends to a $G$-equivariant embedding
	$$ \iota_\pi:L^2_\pi(\Omega)\otimes L^2_\pi(\Omega)^{-\infty,\psi_v}\hookrightarrow L^2(G/N,\psi_v). $$
	Moreover, the map
	$$ \bigoplus_\pi\iota_\pi:\bigoplus_\pi L^2_\pi(\Omega)\otimes L^2_\pi(\Omega)^{-\infty,\psi_v}\hookrightarrow L^2(G/N,\psi_v) $$
	constructs the complete holomorphic discrete series contribution to the Plancherel formula for $L^2(G/N,\psi_v)$.
\end{theorem}

\begin{proof}
	By Frobenius reciprocity, every embedding of $L^2_\pi(\Omega)^\infty$ into $C^\infty(G/N,\psi_v)$ is given in terms of matrix coefficients with respect to a generalized Whittaker vector, so the holomorphic discrete series contribution to $L^2(G/N,\psi_v)$ can be at most what is claimed. Hence, we only need to show that every $\iota_\pi$ indeed extends to an embedding of $L^2_\pi(\Omega)\otimes L^2_\pi(\Omega)^{-\infty,\psi_v}$ into $L^2(G/N,\psi_v)$. For this it suffices to show that for every $W\in L^2_\pi(\Omega)^{-\infty,\psi_v}$ there exists a vector $0\neq f\in L^2_\pi(\Omega)^\infty$ such that $\iota_\pi(f\otimes W)\in L^2(G/N,\psi)$. We use the following integral formula for $G/N$ (see e.g. \cite[Chapter V, \S6]{Kna86}):
	$$ \int_{G/N}\varphi(gN)\,d(gN) = \int_K\int_L \varphi(klN)l^{2\rho}\,dl\,dk, $$
	where $l^{2\rho}=\Delta(l\cdot e)^{\frac{n}{r}}$. By Proposition~\ref{prop:WhittakerVectorsL2}, every Whittaker vector $W$ is of the form $W=W_{v,\eta}^\Omega$ for some $\eta\in\calV_\pi^\vee$. The lowest $K$-type of $L^2_\pi(\Omega)$ is constituted of the functions $f_\xi(u)=e^{-\tr(u)}\xi$ for $\xi\in\calV_\pi$. For these functions we have by Lemma~\ref{lem:ActionLKT} and \eqref{eq:L2ActionL}:
	$$ \iota_\pi(f_\xi\otimes W_{v,\eta}^\Omega)(kl) = W_{v,\eta}^\Omega\big(R_\pi(l^{-1}k^{-1})f_\xi\big) = \eta\big(\pi(l^*)\pi_c(k^{-1})\xi\big) e^{-\tr(l^{-*}v)}. $$
	Therefore, we obtain
	$$ \int_{G/N}|\iota_\pi(f_\xi\otimes W_{v,\eta}^\Omega)(g)|^2\,d(gN) = \int_K\int_L|\eta\big(\pi(l^*)\pi_c(k^{-1})\xi\big)|^2 e^{-2\tr(l^{-*}v)}\Delta(l\cdot e)^{\frac{n}{r}}\,dl\,dk. $$
	The inner integral can be evaluated using Schur's orthogonality relations:
	$$ = \frac{1}{\dim(\pi)}\|\xi\|^2\int_L\|\pi^\vee(l^{-*})\eta\|^2e^{-2\tr(l^{-*}\cdot v)}\Delta(l\cdot e)^{\frac{n}{r}}\,dl, $$
	where we use the norm on $\calV_\pi^\vee$ induced from the norm on $\calV_\pi$. Substituting $l\mapsto l^{-*}$ yields
	$$ = \frac{1}{\dim(\pi)}\|\xi\|^2\int_L\|\pi^\vee(l)\eta\|^2e^{-2\tr(l\cdot v)}\Delta(l\cdot e)^{-\frac{n}{r}}\,dl. $$
	By \cite[proof of Theorem 2.5, especially (5)]{DG93}, the integral is finite if and only if $\omega(\pi^\vee)>\frac{2n}{r}-1$. The statement now follows since $(\pi^\vee,\calV_\pi^\vee)$ has the same lowest restricted weight as $(\pi,\calV_\pi)$.
\end{proof}

\subsection{The formal dimension}

In this section we put for simplicity $v=e$, i.e. $\psi(n_u)=e^{-i(u|e)}=e^{-i\tr(u)}$. Recall the embeddings
$$ \iota_\pi:L^2_\pi(\Omega)\otimes L^2_\pi(\Omega)^{-\infty,\psi}\hookrightarrow L^2(G/N,\psi), \quad \iota_\pi(f\otimes W)(g) = W\big(R_\pi(g)^{-1}f\big). $$
In this section we determine the precise inner product on the space $L^2_\pi(\Omega)^{-\infty,\psi}$ of Whittaker vectors which makes $\iota_\pi$ unitary.

Recall that $\eta\in\calV_\pi^\vee$ parameterizes the space $L^2_\pi(\Omega)^{-\infty,\psi}$ by $\eta\mapsto W_{e,\eta}^\Omega$. We introduce a natural inner product on $L^2_\pi(\Omega)^{-\infty,\psi}$. In fact, if one formally evaluates the inner product on $L^2_\pi(\Omega)$ given by
$$ (f,f')\mapsto\int_\Omega\langle\widetilde{\Gamma}_\pi\pi(u^{\frac{1}{2}})^{-1}f(u),\pi(u^{\frac{1}{2}})^{-1}f(u)\rangle\Delta(u)^{-\frac{n}{r}}\,du $$
at $f=W_{e,\eta}^\Omega=\delta_e\otimes\eta\in\calE'(\Omega)\otimes\calV_\pi^\vee\subseteq L^2_\pi(\Omega)^{-\infty}$ and $f'=W_{e,\eta'}^\Omega$, then one obtains the inner product
\begin{equation}
	\langle W_{e,\eta}^\Omega,W_{e,\eta'}^\Omega\rangle_{L^2_\pi(\Omega)^{-\infty,\psi}} := \langle\widetilde{\Gamma}_\pi^*\eta,\eta'\rangle_{\calV_\pi^\vee},\label{eq:DefInnerProductOnWhittakerVectors}
\end{equation}
where $\widetilde{\Gamma}_\pi^*\eta=\eta\circ\widetilde{\Gamma}_\pi\in\calV_\pi^\vee$ and $\widetilde{\Gamma}_\pi$ denotes the operator-valued gamma function as defined in \eqref{eq:DefGammaFunction}.

\begin{theorem}[Formal dimension]\label{thm:FormalDimension}
	\begin{enumerate}
		\item\label{thm:FormalDimension1} For every holomorphic discrete series representation $(R_\pi,L^2_\pi(\Omega))$ there exists a constant $d(\pi)>0$ called the \emph{formal dimension} such that
		$$ \int_{G/N}W\big(R_\pi(g)^{-1}f\big)\overline{W'\big(R_\pi(g)^{-1}f'\big)}\,d(gN) = \frac{1}{d(\pi)}\langle f,f'\rangle_{L^2_\pi(\Omega)}\langle W,W'\rangle_{L^2_\pi(\Omega)^{-\infty,\psi}} $$
		for all $f,f'\in L^2_\pi(\Omega)$ and $W,W'\in L^2_\pi(\Omega)^{-\infty,\psi}$.
		\item\label{thm:FormalDimension2} The map $\sqrt{d(\pi)}\iota_\pi:L^2_\pi(\Omega)\otimes L^2_\pi(\Omega)^{-\infty,\psi}\to L^2(G/N,\psi)$ is unitary.
		\item\label{thm:FormalDimension3} If $-\frac{1}{2}\sum_{j=1}^rm_j\gamma_j$ denotes the lowest restricted weight of $\pi$ and we write ${\bf m}=(m_1,\ldots,m_r)$, then
		\begin{equation}
			d(\pi) = \const\times\frac{4^{m_1+\cdots+m_r}\dim(\pi)}{\Gamma_\Omega({\bf m})\Gamma_\Omega({\bf m}-\frac{n}{r})},\label{eq:FormulaFormalDimension}
		\end{equation}
		where the constant only depends on the dimension $n=\dim V$, the rank $r=\operatorname{rk}V$ and the structure constant $d$, and $\Gamma_\Omega$ denotes the gamma function of the symmetric cone:
		$$ \Gamma_\Omega({\bf s}) = (2\pi)^{\frac{n-r}{2}}\prod_{j=1}^r\Gamma\left(s_j-(j-1)\frac{d}{2}\right). $$
	\end{enumerate}
\end{theorem}

To prove this, we first compute explicitly the norm of a function in the lowest $K$-type of $L^2_\pi(\Omega)$. Recall that the lowest $K$-type consists of the functions $f_\xi(u)=e^{-\tr(u)}\xi$, $\xi\in\calV_\pi$.

\begin{lemma}\label{lem:NormLowestKType}
	For $\xi\in\calV_\pi$ we have
	$$ \|f_\xi\|_{L^2_\pi(\Omega)}^2 = 2^{n-2(m_1+\cdots+m_r)}\Gamma_\Omega({\bf m})\Gamma_\Omega({\bf m}-\tfrac{n}{r})\|\xi\|_{\calV_\pi}^2. $$
\end{lemma}

\begin{proof}
	Since the lowest $K$-type is irreducible, there exists a constant $c(\pi)>0$ such that
	$$ \|f_\xi\|_{L^2_\pi(\Omega)}^2 = c(\pi)\cdot\|\xi\|_{\calV_\pi}^2 \qquad \mbox{for all }\xi\in\calV_\pi. $$
	We find $c(\pi)$ by computing $\|f_\xi\|_{L^2_\pi(\Omega)}^2$ for a lowest weight vector $\xi\in\calV_\pi$. By definition of the norm in $L^2_\pi(\Omega)$ and the gamma function $\widetilde{\Gamma}_\pi$ we have
	\begin{align*}
		\|f_\xi\|_{L^2_\pi(\Omega)}^2 &= \int_\Omega\langle\widetilde{\Gamma}_\pi\pi(u^{\frac{1}{2}})^{-1}f_\xi(u),\pi(u^{\frac{1}{2}})^{-1}f_\xi(u)\rangle\Delta(u)^{-\frac{n}{r}}\,du\\
		&= \int_\Omega\langle\widetilde{\Gamma}_\pi\pi(u^{\frac{1}{2}})^{-1}\xi,\pi(u^{\frac{1}{2}})^{-1}\xi\rangle e^{-2\tr(u)}\Delta(u)^{-\frac{n}{r}}\,du\\
		&= \int_\Omega\langle\pi(u^{\frac{1}{2}})^{-1}\widetilde{\Gamma}_\pi\pi(u^{\frac{1}{2}})^{-1}\xi,\xi\rangle e^{-2\tr(u)}\Delta(u)^{-\frac{n}{r}}\,du\\
		&= \int_\Omega\int_\Omega\langle\pi(u^{\frac{1}{2}})^{-1}\pi(v)^{-1}\pi(u^{\frac{1}{2}})^{-1}\xi,\xi\rangle e^{-2\tr(u)-2\tr(v)}\Delta(u)^{-\frac{n}{r}}\Delta(v)^{-\frac{2n}{r}}\,dv\,du\\
		&= \int_\Omega\int_\Omega\langle\pi(P(u^{\frac{1}{2}})v)^{-1}\xi,\xi\rangle e^{-2\tr(u)-2\tr(v)}\Delta(u)^{-\frac{n}{r}}\Delta(v)^{-\frac{2n}{r}}\,dv\,du.
		\intertext{By the substitution $x=P(u^{\frac{1}{2}})v$ with $dx=\Delta(u)^{\frac{n}{r}}\,dv$ and $\Delta(x)=\Delta(u)\Delta(v)$ we find}
		&= \int_\Omega\int_\Omega\langle\pi(x)^{-1}\xi,\xi\rangle e^{-2\tr(u)-2(u^{-1}|x)}\Delta(x)^{-\frac{2n}{r}}\,du\,dx. 
	\end{align*}
	Now, if $\xi\in\calV_\pi$ is a lowest restricted weight vector, then by definition of the lowest restricted weight space, we have for $n\in\overline{N}_L$:
	$$ \langle\pi(n\cdot x)^{-1}\xi,\xi\rangle = \langle\pi(x)^{-1}\pi(n)^{-1}\xi,\pi(n)^{-1}\xi\rangle = \langle\pi(x)^{-1}\xi,\xi\rangle $$
	and for $x=\sum_{j=1}^r\lambda_je_j$ we get by \eqref{eq:LowestWeightSpaceAction}:
	$$ \langle\pi(x)^{-1}\xi,\xi\rangle = \lambda_1^{m_1}\cdots\lambda_r^{m_r}\|\xi\|^2. $$
	It follows from \cite[Proposition VI.3.10]{FK94} that	
	$$ \langle\pi(x)^{-1}\xi,\xi\rangle = \Delta_{\bf m}(x)\|\xi\|^2, $$
	where $\Delta_{\bf m}(x)$ is the generalized power function of exponent ${\bf m}=(m_1,\ldots,m_r)$ (see \cite[Chapter VII.1]{FK94} for details). This implies
	\begin{align*}
		\|f_\xi\|_{L^2_\pi(\Omega)}^2 &= \|\xi\|_{\calV_\pi}^2\int_\Omega\int_\Omega\Delta_{\bf m}(x) e^{-2\tr(u)-2(u^{-1}|x)}\Delta(x)^{-\frac{2n}{r}}\,du\,dx.
		\intertext{Interchanging the integrals and computing first the integral over $x$ and afterwards the integral over $u$ using \cite[Proposition VII.1.2]{FK94} gives}
		&= 2^{n-(m_1+\cdots+m_r)}\Gamma_\Omega({\bf m}-\tfrac{n}{r})\|\xi\|_{\calV_\pi}^2\int_\Omega e^{-2\tr(u)}\Delta_{{\bf m}-\frac{n}{r}}(u)\,du\\
		&= 2^{n-2(m_1+\cdots+m_r)}\Gamma_\Omega({\bf m})\Gamma_\Omega({\bf m}-\tfrac{n}{r})\|\xi\|_{\calV_\pi}^2.\qedhere
	\end{align*}
\end{proof}

\begin{proof}[{Proof of Theorem~\ref{thm:FormalDimension}}]
	Since $L_\pi^2(\Omega)$ is irreducible, there exists for every $\eta\in\calV_\pi^\vee$ a constant $d(\pi,\eta)>0$ such that
	$$ \|\iota_\pi(f\otimes W_{e,\eta}^\Omega)\|_{L^2(G/N,\psi)}^2 = \frac{1}{d(\pi,\eta)}\|f\|_{L^2_\pi(\Omega)}^2 \qquad \mbox{for all }\eta\in\calV_\pi^\vee. $$
	Consider this identity for $f$ in the lowest $K$-type, i.e. $f=f_\xi$ for some $\xi\in\calV_\pi$. The right hand side was computed in Lemma~\ref{lem:NormLowestKType}. We now compute the left hand side. By the proof of Theorem~\ref{thm:HolDSContribution} we have
	$$ \int_{G/N} |\iota_\pi(f_\xi\otimes W_{e,\eta}^\Omega)(g)|^2\,d(gN) = \frac{1}{\dim(\pi)}\|\xi\|^2\int_L\|\pi^\vee(l)\eta\|^2e^{-2\tr(l\cdot e)}\Delta(l\cdot e)^{-\frac{n}{r}}\,dl. $$
	Using the integral formula \cite[Section 1.10]{DG93} this equals
	$$ = \frac{1}{\dim(\pi)}\|\xi\|^2\int_\Omega\|\pi^\vee(P(x^{\frac{1}{2}}))\eta\|^2e^{-2\tr(x)}\Delta(x)^{-\frac{2n}{r}}\,dx. $$
	Now,
	$$ \|\pi^\vee(P(x^{\frac{1}{2}}))\eta\|^2 = (\pi^\vee(P(x^{\frac{1}{2}}))\eta|\pi^\vee(P(x^{\frac{1}{2}}))\eta) = (\pi^\vee(P(x))\eta|\eta), $$
	so we find
	\begin{align*}
		&= \frac{1}{\dim(\pi)}\|\xi\|^2\int_\Omega(\pi^\vee(x)\eta|\eta)e^{-2\tr(x)}\Delta(x)^{-\frac{2n}{r}}\,dx\\
		&= \frac{1}{\dim(\pi)}\|\xi\|^2\int_\Omega(\eta\circ\pi(x)^{-1}|\eta)e^{-2\tr(x)}\Delta(x)^{-\frac{2n}{r}}\,dx\\
		&= \frac{1}{\dim(\pi)}\|\xi\|^2(\eta\circ\widetilde{\Gamma}_\pi|\eta) = \frac{1}{\dim(\pi)}\|\xi\|^2(\widetilde{\Gamma}_\pi^*\eta|\eta).
	\end{align*}
	The result now follows from Lemma~\ref{lem:NormLowestKType} and polarization.
\end{proof}

\section{Alternative construction of Whittaker embeddings}\label{sec:NKP}

In this section we give a different approach to explicit embeddings $\calH_\pi^2\hookrightarrow L^2(G/N,\psi)$ of a unitary highest weight representation $\calH_\pi^2$. We first show that $G\subseteq P^\pm K_\CC N_\CC$
(Theorem \ref{thm:Decomp2}) via an $\SU(1,1)$-reduction and then show that 
the integral kernel of an intertwining operator from a unitary highest weight representation into $L^2(G/N,\psi)$ is given in terms of 
the decomposition of an element $g\in G$ as $g=p^\pm k n $ with $n\in N_\CC$, $k\in K_\CC$, $p^\pm \in P^\pm$, 
see Remark \ref{rem:55},
Proposition \ref{prop:EmbeddingKernelPsi}, and Theorem \ref{thm:LKTonG/N}. This is 
in the spirit of \cite{OO88,OO91} where a similar inclusion $G\subseteq P^\pm K_\CC H_\CC $ was derived for an affine symmetric space $G/H$. Finally, we use the explicit formulas for Whittaker embeddings to conclude that 
Theorem~\ref{thm:HolDSContribution} not only constructs the complete holomorphic discrete series 
contribution but the complete unitary highest weight contribution to $L^2(G/N,\psi)$, i.e. the only 
unitary highest weight representations occurring in $L^2(G/N,\psi)$ are the holomorphic discrete series representations.

\subsection{The case $G=\SU(1,1)$}

Let $G=\SU(1,1)$, realized as the subgroup of $\SL(2,\CC)$ fixing the Hermitian form $(z_1,z_2)\mapsto|z_1|^2-|z_2|^2$, and write $\frakg=\su(1,1)$ for its Lie algebra. We fix the Cartan involution $\theta(g)=(g^*)^{-1}$ on $G$ and its differentiated version $\theta(x)=-x^*$ on $\frakg$. Let $K=G^\theta$ and $\frakk=\frakg^\theta$.

Since $\frakg$ is a totally real subspace of $\sl(2,\CC)$, we can complexify subalgebras $\frakh$ of $\frakg$ inside $\sl(2,\CC)$ by $\frakh_\CC=\frakh+i\frakh\subseteq\sl(2,\CC)$. Passing to the corresponding connected analytic subgroups of $\SL(2,\CC)$ constructs natural complexifications on the group level.

Let us fix some notation. We have the Iwasawa decomposition $G=NAK$ with
\begin{align*}
	A &= \{a_t:t\in\RR\}, & a_t &= \exp t\begin{pmatrix} 0 & 1 \\ 1 & 0\end{pmatrix} = \begin{pmatrix} \cosh t & \sinh t\\ \sinh t &\cosh t\end{pmatrix},\\
	N &=  \{n_x:x\in\RR\}, & n_x &= \exp x\begin{pmatrix} i & -i \\ i & -i\end{pmatrix} = \begin{pmatrix} 1+ ix & -ix \\ ix & 1-ix\end{pmatrix}.
\end{align*}
The natural complexification of $N$ is
$$ N_\CC = \{n_z:z\in\CC\}, \qquad n_z = \begin{pmatrix} 1+ iz & -iz \\ iz & 1-iz\end{pmatrix}. $$
On the other hand, we have the decomposition $\frakg_\CC=\frakp^-\oplus\frakk_\CC\oplus\frakp^+$ with
$$ \frakp^+ = \CC\begin{pmatrix} 0 & 1\\ 0 & 0\end{pmatrix} \qquad \mbox{and} \qquad \frakp^- = \CC\begin{pmatrix} 0 & 0\\ 1 & 0 \end{pmatrix}. $$
The corresponding groups are given by
\begin{align*}
	P^+ &= \left\{ p^+_z=\begin{pmatrix} 1 & z \\ 0 & 1\end{pmatrix}:z\in \CC\right\},\\
	K_\CC &= \left\{k_\theta=\begin{pmatrix}e^{i\theta/2} & 0 \\ 0 & e^{-i\theta /2}\end{pmatrix}:\theta\in\CC\right\},\\
	P^- &= \left\{p^-_w=\begin{pmatrix} 1 & 0 \\ w & 1\end{pmatrix}:z\in \CC\right\}.
\end{align*}

\begin{lemma}\label{lem:NKPdecomp1}
Let
$$ g=\begin{pmatrix}a&b\\c&d\end{pmatrix}\in\SL(2,\CC). $$
Then $g=p_x^+ k_\theta n_s  \in P^+ K_\CC N_\CC$ if and only if $c+d\neq 0$. In that case we have
\[e^{i\theta/2} = \frac{1}{c+d}, \quad\text{and}\quad  s = \frac{-ic}{c+d}.\]
In particular $ P^+ K_\CC N_\CC$ is dense in $G_\CC$.
\end{lemma}  
\begin{proof}
We have with $\gamma = e^{i\theta /2}$:
\begin{equation}\label{eq:pkn}
p_x^+k_\theta n_s = \begin{pmatrix} * & *\\ \gamma^{-1}is & \gamma^{-1} (1-is)\end{pmatrix} = 
\begin{pmatrix} a & b\\ c & d\end{pmatrix}
\end{equation}
and the claim follows.
\end{proof}

If $g\in\SL(2,\CC)$ is in $P^+K_\CC N_\CC$ write
\[g= p^+(g)k_\CC (g) n_\CC (g)  .\]

The group $\SL(2,\CC)$ acts on the Riemann sphere $\CC\cup\{\infty\}$ by fractional linear transformations, and with respect to this action, $\SU(1,1)$ acts transitively on the unit disc $D=\{z\in\CC:|z|<1\}\simeq\SU(1,1)/U(1)$. We write $\overline{D}$ for its closure. Define
\begin{equation}\label{eq:XiS}
\Xi = \{s\in \SL(2,\CC) : s\cdot D\subseteq D\}\quad\text{and} \quad
S = \Xi^{-1}.\end{equation}
Then $S$ is a closed semigroup containing $G$ with interior
\[S^o =\{s\in \SL (2,\CC) : s^{-1} \cdot \overline{D} \subseteq D\}\not=\emptyset .\]

\begin{theorem} \label{lem:NKPdecompositionSU11}
 If $g\in \SL (2,\CC)$ maps $\overline{D} $ into $\overline{D}$ then
$g\in P^+K_\CC N_\CC$. In particular,   $\Xi \subseteq P^+ K_\CC N_\CC$. Furthermore,
if $g=a_t$ then
\[k_\CC (g) =\begin{pmatrix} e^{-t} & 0\\ 0 & e^t\end{pmatrix} =k_{2it}   \quad \text{and}\quad n_\CC(g)=n_{i(e^{-2t}-1)/2}.\]
\end{theorem}
\begin{proof}
Let $z\in D$ and suppose that $g=\begin{pmatrix} a & b\\ c & d\end{pmatrix}\in \SL (2,\CC)$ maps $\overline{D}$ into $\overline{D}$. Then
\[g\cdot z = \frac{az + b}{cz+d} \in D,\]
so $\left| \frac{az + b}{cz+d}\right|\le 1$. Assume that $c+d = 0$. Then 
\[\lim_{z\to 1} \left|\frac{az + b}{cz+d}\right| \le 1\]
which implies that $a+b=0$. But then the second column of $g$ is the negative of the first column, which
implies that $\det g =0$, contradicting the assumption that $\det g =1$.
Hence, $c+d\not=0$. The last statement follows from \eqref{eq:pkn}.
\end{proof}

The unique (up to scalar multiples) left-invariant measure on $G/N$ is given by
\begin{equation}\label{eq:GNInt}
\int_{G/N} f(x)d\mu_{G/N} (x) = \frac{1}{4\pi} \int_{0}^{4\pi}\int_{-\infty}^\infty f(k_\theta a_tN) e^{2t}\,dt\,d\theta.
\end{equation}

The holomorphic discrete series of $G$ is parametrized by integers $n>1$ and is realized in the Hilbert space $\mathcal H_n^2(D)$ of
holomorphic functions on $D$ satisfying 
\[ \|f\|_n^2 =\frac{\Gamma (n)}{\pi} \int_D |f(z)|^2(1-|z|^2)^{n-2}\,dx\,dy<\infty, \]
where the inner product is normalized so that $\|1\|_n = 1$. The representation $D_n$ is given by
\[D_n(g)  f(z) = (-cz + a)^{-n}f(g^{-1}z) \quad\text{where } g= 
 \begin{pmatrix} a & b\\ c & d\end{pmatrix} .\] 
 In particular, with $f_k (z) = z^k$ we
 have
 \[D_n (k_\theta) f_k(z) = e^{-i(k +n/2)\theta} f_k(z).\]
 
 The representation $D_n$ extends to a holomorphic representation of $\Gamma= (S^o)^{-1}\subseteq\Xi$. For $n>1$ and $v\in\RR$ let
 $$ \psi_v (n_s) =e^{-i v s} \qquad\mbox{and}\qquad \chi_n(k_{\theta}) = e^{-in\theta/2}. $$
 For $\gamma \in P^+ K_\CC N_\CC $ define
 \begin{equation}\label{eq:Fpi}
 F_n (\gamma ) = \psi_{v}(n_\CC(\gamma )^{-1} )\chi_n(k_\CC (\gamma )^{-1}) .
 \end{equation}
 Then
 \[F_n(k_\theta^{-1} \gamma n) = \psi_{v}(n)^{-1} \chi_n(k_\theta)F_n (\gamma ) \]
 and
 \[|F_n (k_\theta a_t n)|^2 =  e^{-2nt}e^{v(1-e^{-2t})}. \]
 Hence, \eqref{eq:GNInt} implies that in fact $F_n\in L^2(G/N,\psi_v)$ if and only if $n>1$ and $v>0$. It
 follows that $F_n$ belongs to the lowest $K$-type of the holomorphic discrete series $D_n$ realized in
 $L^2(G/N,\psi_v)$.
 
 \subsection{The decomposition $G\subseteq P^+ K_\CC N_\CC$, the general case}

We now let $G$ be as in Section~\ref{sec:AssociatedGroups} and assume that it is contained in a connected complex Lie group $G_\CC$ whose Lie algebra is a complexification of $\frakg$. We form the connected subgroup $N_\CC$ of $G_\CC$ with Lie algebra $\frakn_\CC$. Further, assume that the Cartan involution $\theta$ on $G$ extends to a holomorphic involution on $G_\CC$ and let $K_\CC$ denote its fixed point subgroup. Decompose $\frakp_\CC=\frakp^+\oplus\frakp^-$ as in \eqref{eq:p-plus-minus} and write $P^\pm=\exp\frakp^\pm$. We remark that $N_\CC$ is conjugate to $P^\pm$ inside $G_\CC$ just as $L_\CC$ is conjugate to $K_\CC$.

Recall from Section~\ref{sec:AssociatedGroups} the maximal set of strongly orthogonal roots $\gamma_1,\ldots,\gamma_r$, $r=\operatorname{rank}G/K$, and the corresponding $\sl(2)$-triples $(h_j,e_j,f_j)$. Each of the roots $\gamma_j$ defines a homomorphism $\varphi_j : \SU (1,1)\to G$ determined on the Lie algebra level by
\[ \begin{pmatrix}0&1\\1&0\end{pmatrix}\mapsto h_j, \qquad \frac{1}{2}\begin{pmatrix}i&-i\\i&-i\end{pmatrix}\mapsto e_j, \qquad \text{and} \qquad  \frac{1}{2}\begin{pmatrix}-i&-i\\i&i\end{pmatrix}\mapsto f_j. \]
Furthermore, $\fraka =\bigoplus_{j=1}^r \mathbb{R} h_j$ is maximal abelian in $\mathfrak{p}$. Let $A=\exp \fraka$, then we have $L = (K\cap L)A(K\cap L)$. Note that the groups $G_j = \varphi_j(\SU (1,1))$ are pairwise commuting and
\[G= K  L N = KA (L\cap K) N.\]
 
As before, we use the Harish-Chandra embedding of $G/K$ as a bounded domain $D\subseteq P^+ \subseteq G_\CC /K_\CC P^-_\CC$ and consider the semigroups
\begin{equation}
	\Xi=\{s\in G_\CC : s\cdot D \subseteq D\} \qquad\text{and}\qquad S=\Xi^{-1}.\label{eq:XiSGeneral}
\end{equation}
Then $G\subseteq\Xi$ and the interior of $\Xi$ is given by
 \[\Xi^o = \{s\in G_\CC : s\cdot \overline{D} \subseteq D\}.\]
 Both $\Xi$ and $S$ are $G$-invariant and we have $G=(\Xi \cap S)_o$ and $G\cap \Xi^o=G\cap S^o=\emptyset$.
 
\begin{theorem}\label{thm:Decomp2}
	We have $\Xi\subseteq P^+ K_\CC  N_\CC  $. Moreover, the decomposition of $g\in \Xi$ into 
$g=p^+ (g) k_\CC (g) n_\CC(g) $ with $n_\CC(g)\in N_\CC$, $k_\CC(g)\in K_\CC$ and $p^+(g)\in P^+$ is unique.
\end{theorem}

\begin{proof}
	We first show that $G\subseteq P^+ K_\CC  N_\CC$. Since $G=KLN $ and $P^+ K_\CC N_\CC $ is left $K_\CC$-invariant
and right $(K_\CC\cap L_\CC) N_\CC$-invariant, it suffices to show that $A\subseteq P^+ K_\CC N_\CC $. 
But that follows from Theorem \ref{lem:NKPdecompositionSU11} applied to the subgroups $G_j=\varphi_j(\SU(1,1))$.

Now denote by $: G_\CC \to G_\CC,\,g\mapsto\overline{g}$ the conjugation with respect to $G$ and let $s^* = \overline{s}^{-1}$. Then
$\Xi^* = \Xi$ and $(P^-)^* = P^+$. Let $s\in \Xi$, then $s^*\cdot 0 \in D$ and hence there exists $g\in G$ such that $g^{-1}s^*\cdot0=0$. The stabilizer of $0$ in $G_\CC$ is $K_\CC P^-$, so we can write $g^{-1}s^*=kp^-$ for some $k\in K_\CC$ and $p^- \in P^-$. Thus, using $G\subseteq P^+ K_\CC  N_\CC$, we find
\[s=(p^-)^*k^* g^{-1} \in P^+ K_\CC G\subseteq P^+ K_\CC N_\CC .\]
	
The second claim is a consequence of the fact that $K_\CC P^\pm\simeq K_\CC\ltimes P^\pm$ is a group and $N_\CC\cap K_\CC P^\pm=\{e\}$.
\end{proof}

\subsection{Intertwining operators $\calH^2_\pi\to L^2(G/N,\psi)$}
Let $(J_\pi,\calH_\pi^2)$ be a unitary highest weight representation of $G$ with highest weight space $(\pi,\calV_\pi)$, realized on a space of holomorphic functions $f:G/K\to\calV_\pi$. Recall that $J_\pi$ extends to a representation of $S$ such that $J_\pi(s^*) =J_\pi(s)^*$. Identifying $G/K$ with either the bounded symmetric domain $D$ or the tube domain $T_\Omega$ yields the models discussed in Sections~\ref{sec:TubeDomainModel} and \ref{sec:BoundedDomainModel}. Note that in the case of the unbounded realization we have to replace $D$ by $T_\Omega$ in \eqref{eq:XiSGeneral}. The action is in both cases given by
$$ J_\pi(g)f(z) = j_\pi(g^{-1},z)f(g^{-1}z) \qquad (g\in S, z\in G/K), $$
for a cocycle $j_\pi:\Xi\times G/K\to\GL(\calV_\pi)$ satisfying
$$ j_\pi(gh,z) = j_\pi(h,z)j_\pi(g,hz) \qquad (g,h\in\Xi ,z\in G/K). $$

Now assume that $T:\calH_\pi^2\to L^2(G/N,\psi)$ is a unitary intertwining operator. Let $f\in L^2(G/N,\psi)$ and $z\in G/K$, then
\[ \langle T^*f(z),\xi\rangle = \ip{T^*f}{K_z\xi} =\ip{f}{TK_z\xi}=\int_{G/N}f(g)\overline{T(K_z\xi)(g)}\,d\dot{g} \]
for all $\xi\in\calV_\pi$, where $K_z:G/K\to\GL(\calV_\pi)$ denotes the reproducing kernel of $\calH_\pi^2$. Let
$K_{z,\xi}(w) = K_z(w)\xi$. Define $\Psi:G\times G/K\to\calV_\pi$ by
\begin{equation}
	\langle\Psi(g,z),\xi\rangle = \overline{T(K_{z,\xi})(g)} \qquad (g\in G,z\in G/K,\xi\in\calV_\pi),\label{eq:TvsPsi}
\end{equation}
then
$$ T^*f(z) = \int_{G/N}f(g)\Psi(g,z)\,d\dot{g}. $$
Note that this is well defined because
\begin{equation} \label{eq:Ps1}
\Psi (gn,z) = \overline{\psi(n)^{-1}}\Psi (g , z) = \psi (n) \Psi (g , z).
\end{equation}

We note that, since $K_z\in(\calH_\pi^2)^\infty$ and $L^2(G/N,\psi)^\infty\subseteq C^\infty(G/N,\psi)$, the function $\Psi$ is smooth. Moreover, since $z\mapsto K_z$ is antiholomorphic, the function $\Psi(g,z)$ is holomorphic in $z$. Moreover, as $T$ is
an intertwining operator, we can rewrite
$$ \langle\Psi(g,z),\xi\rangle = \overline{(T\circ J_\pi(g^{-1})K_{z,\xi})(e)} =  \overline{(T\circ J_\pi(g^*)K_{z,\xi})(e)} $$
and conclude that $\Psi(g,z)$ extends to a continuous function in $g\in S$ holomorphic on $S^o$. 

\begin{lemma}\label{le:TrPsi}
	\begin{enumerate}
		\item\label{le:TrPsi1} $\Psi(gn,z)=\psi(n)\Psi(g,z)$ for all $g\in S$, $n\in N$ and $z\in G/K$.
		\item\label{le:TrPsi2} $\Psi(gx,z)=j_\pi(g^{-1},z)\Psi(x,g^{-1}z)$ for all $g,x\in S$ and $z\in G/K$.
	\end{enumerate}	
\end{lemma}

\begin{proof}
	The first is \eqref{eq:Ps1}. To show the second identity, we use that $T$, and hence also $T^*$, is an intertwining operator. Thus, for $f\in L^2(G/N,\psi)$, $g\in G$ and $z\in G/K$:
	\begin{align*}
		& \int_{G/N} f(x)j_\pi (g^{-1},z )\Psi (x,g^{-1}z)d\mu_{G/N} (x) = j_\pi (g^{-1},z) T^*f(g^{-1}z)\\
		={}& [J_\pi(g)\circ T^*f](z) = [T^*\circ L_gf](z) = \int_{G/N} f(g^{-1}x)\Psi (x , z)\,d\mu_{G/N}(x)\\
		={}& \int_{G/N} f(x)\Psi(gx,z)\,d\mu_{G/N}(x),
	\end{align*}
	where we have tacitly used that $x\mapsto f(x)\Psi (x,z)$ is right $N$-invariant. As this holds for all $f\in L^2(G/N,\psi)$, it follows that
	\begin{equation}\label{eq:Pg}
		\Psi (gx,z) = j_\pi (g^{-1},z) \Psi (x, g^{-1}z).
	\end{equation}
As both sides are holomorphic on $S^o$ and continuous on $S$, it follows that \eqref{eq:Pg} holds on all of $S$.
\end{proof}

\begin{remark}\label{rem:55}
	\begin{enumerate}
		\item The kernel $\Psi$ is uniquely determined by the function
		\[ \Phi:G/K\to\calV_\pi, \quad \Phi (z)=\Psi (e,z). \]
		More precisely, by Lemma~\ref{le:TrPsi}~\eqref{le:TrPsi2}:
\[
\Psi (g , z) = j_\pi (g^{-1},z)\Phi (g^{-1}z) \qquad (g\in S,z\in G/K). 
\]
		\item One could just as well use the function
		\[ \widetilde{\Phi}:G\to\calV_\pi, \quad \widetilde{\Phi} (g)= \Psi (g,eK). \]
		Then Lemma~\ref{le:TrPsi}~\eqref{le:TrPsi2} shows that
		\[\Psi (g,xK) =j_\pi (x,eK)^{-1} \widetilde{\Phi} (x^{-1}g) .\]
Furthermore, by Lemma~\ref{le:TrPsi} the function $\widetilde{\Phi}$ satisfies
		\begin{itemize}
			\item $\widetilde{\Phi} (kg)=j_\pi(k^{-1},eK) \widetilde{\Phi} (g)=\pi_c(k)\widetilde{\Phi}(g)$ for $k\in K$,
			\item $\widetilde{\Phi}(gn)= \psi(n)\widetilde{\Phi} (g)$ for $n\in N$,
			\item  $\widetilde{\Phi} (p^-g)=j_\pi((p^-)^{-1},eK)\Psi(g,(p^{-})^{-1}K)=\widetilde{\Phi} (g)$ for $p^-\in P^-$.
		\end{itemize}
Hence, $\widetilde{\Phi}$ is determined by $\widetilde{\Phi}(e)\in\calV_\pi$. In fact, the above equivariance properties can be used to give an explicit formula for $\widetilde{\Phi}$. For this, note that $\Xi \to S,\,\xi\mapsto\overline{\xi}$ is a bijective homomorphism. As $\overline{P^+} = P^-$, it follows from
Theorem \ref{thm:Decomp2} that every element in $S$ has a unique decomposition of the form
$$ s = p^- (s)\widetilde k_{\CC}(s)\widetilde n_\CC(s) \qquad (p^-(s)\in P^-,\widetilde{k}_\CC(s)\in K_\CC,\widetilde{n}_\CC(s)\in N_\CC). $$
Then, for $s\in S$ we have
		\[\widetilde{\Phi} (s) = \widetilde{\Phi} (p^-(s)\widetilde k_\CC(s) \widetilde n_\CC(s) ) = 
	\psi(\widetilde n_\CC(s)) \pi_c(\widetilde k_\CC(s))\widetilde{\Phi} (e). \]
	\end{enumerate}
\end{remark}
Using the explicit calculation in Theorem  \ref{lem:NKPdecompositionSU11} one sees by $\SU (1,1)$ reduction that $\widetilde \Phi$ is
square-integrable if and only if the highest weight representation with lowest $K$-type $\pi$ is in the
holomorphic discrete series. This leads to: 

\begin{proposition}\label{prop:EmbeddingKernelPsi}
	Let $T:\calH_\pi^2\to L^2(G/N,\psi)$ be a $G$-equivariant continuous linear operator. Then there exists a vector $\xi\in\calV_\pi$ such that $T^*:L^2(G/N,\psi)\to\calH_\pi^2$ is given by
\[ T^*f(z) = \int_{G/N} f(x)\Psi(g,z)d\mu_{G/N}(x)\]
where 
\begin{equation}\label{eq:Psi3}
\Psi:G\times G/K\to\calV_\pi, \quad \Psi (a,bK) = \psi(\widetilde n_\CC(b^{-1}a)) j_\pi (b,eK)^{-1} \pi_c(\widetilde k_\CC(b^{-1}a))\xi. 
\end{equation}
Conversely, a kernel $\Psi$ defined by \eqref{eq:Psi3} defines an intertwining operator $T^* :L^2(G/N,\psi )
\to \calH_\pi^2(G/K)$ if and only if $\calH_\pi^2$ belongs to the holomorphic discrete series. In particular, 
$$ \dim \Hom_G (\mathcal H_\pi^2,L^2(G/N,\psi ))= \dim \mathcal V_\pi. $$
\end{proposition}

 \begin{remark}
	The formula in Proposition~\ref{prop:EmbeddingKernelPsi} contains all information about the embedding of a holomorphic discrete series into $L^2(G/N,\psi)$. For instance, one can conclude that the embedding $T:\calH^2_\pi\to L^2(G/N,\psi)$ is given by
	\begin{align*}
		TF(g) &= \langle F,\Psi(g,\cdot)\rangle_{\calH^2_\pi} = \langle F,j_\pi(g,g^{-1}\cdot)^{-1}\Psi(e,g^{-1}\cdot)\rangle_{\calH^2_\pi}\\
		&= \langle F,j_\pi(g^{-1},\cdot)\Psi(e,g^{-1}\cdot)\rangle_{\calH^2_\pi} = \langle F,J_\pi(g)\Psi(e,\cdot)\rangle_{\calH^2_\pi}\\
		&= \langle J_\pi(g)^{-1}F,\Psi(e,\cdot)\rangle
	\end{align*}
	for $g\in G,F\in\calH^2_\pi$. Comparing this expression with the embedding \eqref{eq:DefIotaPi} in terms of Whittaker vectors, we recover the formulas of Proposition~\ref{prop:WhittakerVectorsTOmega}~\eqref{prop:WhittakerVectorsTOmega1} and Proposition~\ref{prop:WhittakerVectorsD}~\eqref{prop:WhittakerVectorsD1} for the Whittaker vectors $W_{v,\eta}=W_{v,\eta}^{T_\Omega}$ and $W_{v,\eta}=W_{v,\eta}^D$. In fact, using the antilinear isomorphism $\calV_\pi\simeq\calV_\pi^\vee,\,\xi\mapsto\overline{\xi}:=\langle\cdot,\xi\rangle$, we obtain that the antiholomorphic function
	$$ W_{v,\eta}(z) = \langle\cdot,\Psi(e,x)\rangle = \overline{\psi(\widetilde{n}_\CC(x^{-1}))}\cdot\overline{j_\pi(x,eK)^{-1}\pi_c(\widetilde{k}_\CC(x^{-1}))\xi} \qquad (z=xK) $$
	is a Whittaker vector for $\calH_\pi^2$. For instance, if we identify $G/K\simeq T_\Omega$, then $z\in T_\Omega$ can be written as $z=n_{z-ie}\cdot(ie)$ with $n_{z-ie}\in N_\CC$, so $z\in T_\Omega$ corresponds to $n_{z-ie}K\in G/K$ and hence
	$$ W_{v,\eta}(z) = \overline{\psi(n_{z-ie}^{-1})} \overline{\xi} = \overline{e^{-i(ie-z|v)}} \overline{\xi} = \const\times e^{-i(\overline{z}|v)}\overline{\xi}. $$
	This is precisely \eqref{eq:WhittakerTOmega} with $\eta=\overline{\xi}$.
\end{remark}

\subsection{A formula for the lowest $K$-types}

In the proof of Theorem~\ref{thm:HolDSContribution} we found a formula for the lowest $K$-type of a holomorphic discrete series embedded into $C^\infty(G/N,\psi)$ and used this formula to show that the lowest $K$-type (and hence the whole holomorphic discrete series) is contained in $L^2(G/N,\psi)$. We now provide another formula for the lowest $K$-type in the spirit of \cite{OO91}, using the decomposition $g=p^+(g)k_\CC(g)n_\CC(g)\in P^-K_\CC N_\CC$. This generalizes the formula \eqref{eq:Fpi} for $\SU (1,1)$. For this we first note that the lowest $K$-type in $\calH_\pi^2$ is spanned by the functions $K_e\xi\in\calH_\pi^2$ ($\xi\in\calV_\pi$), where $K_w(z)=K(z,w)\in\GL(\mathcal V_\pi)$ denotes the reproducing kernel of $\calH_\pi^2$ (see e.g. Section~\ref{sec:TubeDomainModel} for the reproducing kernel of $\calH_\pi^2(T_\Omega)$).

\begin{theorem}\label{thm:LKTonG/N}
	Let $T:\calH^2_\pi\to L^2(G/N,\psi)$ be an intertwining operator. Then the restriction of $T$ to the lowest $K$-type $\calV_\pi$ is given by
	\begin{multline*}
		T_{\pi,\eta}:\calV_\pi\to L^2(G/N,\psi),\\
		T_{\pi,\eta}\xi(g) = \psi(n_\CC(g)^{-1})\eta\big(\pi_c(k_\CC(g)^{-1})\xi\big) \qquad (\xi\in\mathcal V_\pi,g\in G)
	\end{multline*}
	for some $\eta\in\calV_\pi^\vee$. Moreover, the function $T_{\pi,\eta}\xi$ extends to a holomorphic function on $\Xi$. Conversely, for every lowest $K$-type $(\pi,\calV_\pi)$ of a holomorphic discrete series representation and every $\eta\in\calV_\pi^\vee$, the operator $T_{\pi,\eta}$ constructs the lowest $K$-type of a holomorphic discrete series occurring discretely in $L^2(G/N,\psi)$.
\end{theorem}

\begin{proof}
	By Proposition~\ref{prop:EmbeddingKernelPsi}, the integral kernel $\Psi$ of $T^*:L^2(G/N,\psi)\to\calH^2_\pi$ is given by
	$$ \Psi(a,bK) = \psi(\widetilde n_\CC(b^{-1}a)) j_\pi (b,eK)^{-1} \pi_c(\widetilde k_\CC(b^{-1}a))\eta $$
	for some $\eta\in\calV_\pi$. Then, by \eqref{eq:TvsPsi}, a vector $K_e\xi$ ($\xi\in\calV_\pi$) in the lowest $K$-type is is mapped to
	$$ TK_e\xi(g) = \langle\xi,\Psi(g,e)\rangle = \overline{\psi(\widetilde{n}_\CC(g))}\langle\xi,\pi_c(\widetilde{k}_\CC(g))\eta\rangle, $$
	where we have used that $j_\pi(e,eK)=\id$. Identifying $\eta\in\calV_\pi$ with the functional $\xi\mapsto\langle\xi,\eta\rangle$ in $\calV_\pi^\vee$ and using $\overline{\psi(\widetilde{n}_\CC(g))}=\psi(n_\CC(g)^{-1})$ and $\pi_c(\widetilde{k}_\CC(g))^*=\pi_c(k(g)^{-1})$ shows the claim.
\end{proof}

\begin{remark}
	One can also prove Theorem~\ref{thm:LKTonG/N} without using the previous results on the kernel $\Psi$. In fact, the projections $p^+(g)$, $k_\CC(g)$ and $n_\CC(g)$ satisfy
	\[p^+(k g n ) = kp^+(g)k^{-1},\quad k_\CC (k g n) = k k_\CC (g)\quad\text{and}\quad n_\CC (k g n) = n_\CC (g)n,\]
	implying that
	\[T_{\pi,\eta}(\xi)(k^{-1}gn) = \psi(n)^{-1} T_{\pi,\eta}(\pi_c(k)\xi ).\]
	Hence, $T_{\pi,\eta}(\xi) \in C^\infty(G/N,\psi)$ and $T_{\pi,\eta}$ is a $K$-intertwining operator. Moreover, an $\SU (1,1)$ reduction shows that $T_{\pi,\eta}(\xi)$ is square-integrable. That $T_{\pi,\eta}(\xi )$ extends to a function on $\Xi$ holomorphic on $\Xi^o$ follows from Theorem~\ref{thm:Decomp2}. Finally, $T_{\pi,\eta}(\xi)$ is a highest weight vector since $\frakp^+$ acts trivially. This follows from the $P^+$-invariance of the projections $k_\CC(g)$ and $n_\CC(g)$.
\end{remark}

\begin{corollary}\label{cor:UnitHWContribution}
	Every unitary highest weight representation that occurs in the Plancherel formula for $L^2(G/N,\psi)$ is a holomorphic discrete series representation and hence given in Theorem~\ref{thm:HolDSContribution}.
\end{corollary}

\begin{proof}
First note that since the unitary highest weight representations form a discrete 
subset of $\widehat{G}$, they can only occur discretely in $L^2(G/N,\psi)$ 
(see also \cite[Section 3]{HOO91}). For a fixed irreducible finite-dimensional 
representation $(\pi,\calV_\pi)$ of $L_\CC$ and $\pi_\nu=\pi\otimes\det^\nu$ as 
in Section~\ref{sec:UnitaryHighestWeightReps}, the formula in Theorem~\ref{thm:LKTonG/N} is 
clearly holomorphic in $\nu$ and hence agrees with the formula for the lowest $K$-type computed in 
the proof of Theorem~\ref{thm:HolDSContribution}. But there it was shown that the lowest 
$K$-type $\pi_\nu$ is $L^2$ on $G/N$ if and only if $\omega(\pi_\nu)>\frac{2n}{r}-1$, i.e. if 
and only if $\pi_\nu$ is a holomorphic discrete series.
\end{proof}

\section{The twisted Hardy space on $L^2(G/N)$}\label{sec:HardySpace}

In this section we discuss the twisted Hardy space of holomorphic $L^2$-sections of the line bundle 
defined by $\psi=\psi_v$. The idea of realizing series of representations in a Hardy space, the 
Gelfand--Gindikin program, goes back to 
the article \cite{GG77} and was then taken up in \cite{O91,S83} for the case of a semisimple Lie 
group. The first step towards Hardy spaces on symmetric spaces was \cite{OO88}, a program that 
was completed in \cite{HOO91}. Since then, several other situations have been considered, in particular 
the case of Riemannian symmetric spaces was studied in \cite{GKO04}, and the space of horocycles $G/MN_{\rm min}$, $P_{\rm min} =M A_{\rm min}N_{\rm min}$ a minimal parabolic subgroup, in \cite{GKO06}. The discussion is brief because the methods and ideas are the
same as in \cite{HOO91} and \cite[Chapter 9]{N00}. We refer to the following Appendix for a more detailed discussion about
some of those ideas.

We realize $G/K$ again as a bounded domain $D\subseteq P^+\subseteq G_\CC/K_\CC P^-$ and define $\Xi$ and $S$ as before, see \eqref{eq:XiSGeneral}. As before, we define $s\mapsto s^* = \overline{s}^{-1}$ where
$G_\CC \to G_\CC,\,g\mapsto\overline{g}$ is the conjugation fixing $G$ pointwise. A continuous representation $\sigma$ of 
$S$ on a Hilbert space $\calH_\sigma$ is said to be a $*$-representation if for all $s\in S$ we have $\sigma(s)^* = \sigma (s^*)$. It is
a holomorphic representation if $s\mapsto \sigma (s)u$ is a holomorphic map $S^o\to \calH_\sigma$ for every $u\in\calH_\sigma$.

The following can be found in \cite[Thm. XI.6.1]{N00}:

\begin{theorem}\label{thm:HolomExtRep}  
A continuous representation $\sigma$ of $S$  is a holomorphic $*$-representation if and only if $\sigma|_G$ is a 
unitary highest weight representation of $G$. 
\end{theorem}
 
Motivated by the construction in \cite{HOO91} (see also the Appendix), we define the twisted Hardy space $\calH_2(\Xi,\psi_v)$ over $G/N$ to 
be the Hilbert space of holomorphic functions $F$ on $\Xi$ satisfying:
\begin{enumerate}
	\item For all $\xi \in \Xi$ and $n\in N$ we have
	$F (\xi n) = \psi_v(n)^{-1}F(\xi )$.
	\item For every $s\in S^o$ the function $g\mapsto F_s(g)=F(s^{-1}g)$ is in $L^2(G/N,\psi_v)$ and
	the $L^2$-limit
	\[ \beta(F) = \lim_{S^o\ni s\to g}F_s \in L^2(G/N,\psi_v) \]
	exists in $L^2(G/N,\psi_v)$.
\end{enumerate}
The inner product on $\calH_2(\Xi,\psi_v)$ is given by
$$ \langle F,G\rangle_{\calH_2(\Xi,\psi_v)} = \langle\beta(F),\beta(G)\rangle_{L^2(G/N,\psi_v)}. $$

The closed semigroup $S$ acts on $\calH_2 (\Xi,\psi_v)$ by
\[s\cdot F(\xi ) = F(s^{-1}\xi) \qquad (s\in S,\xi\in\Xi)\]
and the boundary value map
$$ \beta:\calH_2(\Xi,\psi_v)\to L^2(G/N,\psi_v) $$
is a $G$-equivariant isometry. We also note that
$\calH_2(\Xi,\psi_v)\not=\{0\}$ because each of the functions $T_{\pi,\eta}(\xi)$ is in $\calH_2(\Xi,\psi_v)$ and
$\beta (T(\xi )) = T_{\pi,\eta}(\xi)|_{G/N}$. Each of those functions generates an $S$-invariant subspace in $\calH_2(\Xi,\psi_v)$.
The representation of $S$ in $\calH_2(\Xi,\psi_v)$ is holomorphic. Hence, Theorem~\ref{thm:HolDSContribution}, Corollary~\ref{cor:UnitHWContribution} and Theorem~\ref{thm:HolomExtRep} imply that
the Hardy space is a direct sum of holomorphic discrete series, each occurring with multiplicity equal to the dimension of the lowest $K$-type.
To be more specific, the following construction applies to all of the
holomorphic discrete series $(J_\pi,\calH_\pi^2)$ leading to Theorem \ref{thm:HardyDec}:
We extend the matrix coefficient map 
$\iota_\pi$ from \eqref{eq:DefIotaPi} to $\calH_2(\Xi,\psi_v)$. For $f\in\calH^2_\pi$ and $W\in(\calH_\pi^2)^{-\infty,\psi_v}$ we let
\[ \iota_\pi(f\otimes W)(\xi) = W(J_\pi(\xi^{-1})f), \quad \xi \in \Xi= S^{o-1}, \]
using that by Lemma~\ref{lem:SemigroupMapsIntoSmoothVectors} we have $J_\pi(\xi^{-1})f\in(\calH_\pi^2)^\infty$. Note that the function $\iota_\pi(f\otimes W)$ satisfies the covariance relation
\begin{align*}
	\iota_\pi(f\otimes W) (\xi n)  & = W(J_\pi(n)^{-1}J_\pi(\xi^{-1})f) = \psi_v(n)^{-1}\iota_\pi(f\otimes W)(\xi).
\end{align*}

\begin{theorem}\label{thm:HardyDec}
	Let the notation be as above. Then the following holds:
	\begin{enumerate}
		\item If $\omega(\pi) > \frac{2n}{r}-1$, then
		$\iota_\pi(\calH^2_\pi\otimes(\calH^2_\pi)^{-\infty,\psi_v}) \subseteq \calH_2(\Xi,\psi_v)$ and the map $\iota_\pi:\calH^2_\pi\otimes(\calH^2_\pi)^{-\infty,\psi_v} \to \calH_2(\Xi,\psi_v)$ is
		an $S$-intertwining operator.
		\item We have
		\[\calH_2 (\Xi ,\psi_v) = \widehat{\bigoplus_{\pi}}\,\,\iota_\pi(\calH^2_\pi\otimes(\calH^2_\pi)^{-\infty,\psi_v}),\]
		where the sum is taken over all $\pi$ with $\omega(\pi)> \frac{2n}{r} -1$
	\end{enumerate}
\end{theorem}

\appendix

\section{Hardy spaces}\label{App:Hardy}

To put our results in Section~\ref{sec:HardySpace} into perspective, we provide a general discussion about the Hardy space of a semisimple Hermitian Lie group and the Hardy space of a symmetric space of Hermitian type. We refer to \cite{HOO91, O00} 
and the book \cite{N00}, in particular Chapters XI, XII and  XIV, for details. The notation is as in the main part of the paper.\\

We start by discussing the original idea \cite{GG77} based on the ideas from \cite{O91}. For a function $F:  \Xi^o \to \CC$ let
 $F_s (g) =F (s^{-1}g)$.
Define the Hardy space on $\Xi$ to be the Hilbert space
\[ \calH_2(\Xi) = \{F\in \calO (\Xi^o ) \mid (\forall s  \in S^o) \, F_s \in L^2(G)\text{ and }
\lim_{S^o\ni s\to e}F_{s}  \text{ exists in }L^2(G)\}. \]
Note that $S$ acts on $\calH_2(\Xi )$ by
$s\cdot F(\xi ) = F(s^{-1}\xi)$. We define the boundary value map $\beta : \calH_2(\Xi )\to L^2(G)$ by 
\[\beta (F) = \lim_{S^o\ni s\to e} F_s.\]

\begin{lemma}\label{le:holods}
Let $\pi$ be so that $\omega (\pi)>\frac{2n}{r}-1$ and denote by $(J_\pi,\calH_\pi^2)$ the holomorphic discrete series representation with lowest $K$-type $\pi$.
	\begin{enumerate}
		\item The representation $J_\pi$ extends to a holomorphic ${}^*$-representation of $S$.
		\item For every $u\in \calH_\pi^2 $ and $\lambda \in (\calH_\pi^2)^*$, the function $\Pi_\pi (u\otimes \lambda)(g) =\lambda (J_\pi(g^{-1})u)$ extends to a holomorphic function on $\Xi$ which is contained in $\calH_2(\Xi)$, and $\beta(\Pi_\pi(u\otimes \lambda))\in L^2 (G)$.
		\item For a fixed $\lambda \in(\calH_\pi^2)^*$, the map $u\mapsto\Pi_\pi(u\otimes\lambda)$
		is an $S$-homorphism $\calH_\pi^2 \to \calH_2(\Xi)$.
		\item $\calH_2 (\Xi) = \widehat{\bigoplus}_{\pi} \Pi_\pi (\calH_\pi^2 \otimes (\calH_\pi^2)^*)$.
		\item The contribution of the holomorphic discrete series to $L^2(G)$ is
		\[\beta (\calH_2 (\Xi))= \widehat{\bigoplus_\pi}\,\Pi_\pi\big(\calH_\pi^2\otimes(\calH_\pi^2)^*\big)|_G .\]
	\end{enumerate}
\end{lemma}

As mentioned above, the next step was the extension of this theory to compactly causal symmetric spaces, originally also
called symmetric spaces of Hermitian type. For that, the following is helpful, even if it was not used in
the original work.

\begin{lemma}\label{lem:SemigroupMapsIntoSmoothVectors}
	Let $s \in S^o $ and $u\in\calH_\pi^2$. Then $J_\pi(s) u\in(\calH_\pi^2)^{\infty } $.
\end{lemma}

\begin{proof}
	By \cite[Proposition A.5]{KNO97} we have $J_\pi(s)\calH_\pi^2\subseteq(\calH_\pi^2)^\omega$, and $(\calH_\pi^2)^\omega\subseteq(\calH_\pi^2)^\infty$.
\end{proof}

Recall that the group $G$ acts on $(\calH_\pi^2)^{-\infty}$ by
\[(J_\pi^{-\infty} (g)\lambda)(u) = \lambda (J_\pi (g^{-1})u),\quad u\in (\calH_\pi^2)^\infty, g\in G.\]
This action extends to $S$ by
\[(J_\pi^{-\infty}(s)\lambda)(u) = \lambda (J_\pi (s^*)u)\]
and the function $\Xi^o \to \CC$, $\xi \mapsto \lambda (J_\pi(\overline{\xi} )u)$, is anti-holomorphic.\\

Now assume that $G/H$ is a symmetric space such that the tangent space $T_{eH}(G/H)$ contains an open, generating 
$H$-invariant elliptic cone. Thus, by denoting the corresponding involution on $G$ and its Lie algebra by $\tau $ and writing
$\frakg = \frakq\oplus \frakh$ with $\frakq = \frakg^{-\tau}$ and $\frakh = \frakg^\tau$, the 
assumption is that $\frakq$ contains an open generating
$H$-invariant cone $C$ such that the spectrum of $\ad X$, $X\in C$, is contained in $i\RR$. The cone $C$ then
extends to a $G$-invariant cone in $\frakg$ and $G/K$ is a bounded 
symmetric domain. In particular, the semigroup $S$ is defined. Assume that $H$ is a real form of a complex group $H_\CC\subseteq G_\CC$. Let $x_o = eH_\CC \in G_\CC/H_\CC$ and
define $\Xi_H=\Xi x_o =\Xi H\subseteq G_\CC/H_\CC$. Then $\Xi_H^o = \Xi^ox_o$ is a $G$-invariant open domain in
$G_\CC/H_\CC$ containing the symmetric space $G/H$ in the boundary. For a holomorphic function $F$ on $\Xi_H^o$ let 
$s\cdot F (gH) = F_s (gH) = F(s^{-1}gH)$. We can now define the Hardy space $\calH_2(\Xi_H)$ as the space of holomorphic functions $F$ on $\Xi_H$ such that
\begin{enumerate}
	\item $\forall s\in S^o$ the function $ s\cdot F\in L^2(G/H)$,
	\item $\sup_{s\in S^o} \|s \cdot F\|_{L^{2}(G/H)}<\infty$.
\end{enumerate}

Then the boundary value map $\beta : \calH_2 (\Xi_H) \to L^2(G/H)$ is defined by
\[\beta (F) = \lim_{S^o\ni s\to e}  F_s, \]
where the limit exists in $L^2(G/H)$, and $\calH_2(\Xi_H)$ is a Hilbert space with inner product
\[\ip{F}{G}_{\calH_2(\Xi_H)} = \ip{\beta (F) }{\beta (G)}_{L^2(G/H)}.\]

For $u\in \calH_\pi^2$ and $\lambda \in (\calH_\pi^2)^{-\infty,H}$ let
\[\Pi_{H,\pi,\lambda} (u)(\xi ) = \lambda (J_\pi (\xi^{-1} )u),\quad \xi \in \Xi.\]
Then for $h\in H$
\begin{align*}
	\Pi_{H,\pi,\lambda} (u)(\xi h) & = \lambda (J_\pi (h^{-1})J_\pi (\xi ^{-1}) u) = (J_\pi^{-\infty} (h)\lambda )(J_\pi (\xi^{-1})u) = \Pi_{H,\pi,\lambda} (u)(\xi ) 
\end{align*}
as $\lambda$ is $H$-invariant. The function $\Pi_{H,\pi,\lambda} (u)$ is holomorphic and right $H$-invariant, hence also right $H_\CC$-invariant, so $\Pi_{H,\pi,\lambda}(u)$ defines a holomorphic function on $\Xi_H$. We also use the notation $\Pi_{H,\pi,\lambda} (u) (g)=\Pi_{H,\pi,\lambda}(u) (gH)$ for $u\in\calH_\pi^2$ and $g\in G$.

The following lemma is first proved first for $\SU (1,1)$ and then generalized to $G_\CC/K_\CC$ using $\SU(1,1)$ reduction.
\begin{lemma} We have $\Xi \subset P^+ K_\CC H_\CC$.
\end{lemma}

For $x\in P^+K_\CC H_\CC$ choose $k_H (x) \in K_\CC$ such that $x\in P^+ k_H(x) H_\CC$. Note that $k_H(x)$ is not unique, but the coset $k_H(x)(H_\CC \cap K_\CC)\in K_\CC/(H_\CC \cap K_\CC)$ is unique and the map $x\mapsto k_H (x)(H_\CC \cap K_\CC)$ is holomorphic.

We can now state the main results on the Hardy space on $\Xi_H$:

\begin{theorem}
	Let $(J_\pi,\calH_\pi^2)$ be a unitary highest weight representation. Then the following holds:
	\begin{enumerate}
		\item $\dim(\calH_\pi^2)^{-\infty,H}\leq1$.
		\item There exists an explicit constant $c_{G/H}>0$ such that for $0\neq\lambda \in (\calH_\pi^2)^{-\infty,H}$ we have
		\[ \Pi_{H,\pi,\lambda}(\calH_\pi^2) \subseteq \calH_2(\Xi_H) \quad\Leftrightarrow\quad \omega(\pi) >c_{G/H}, \]
		and in that case $\Pi_{H,\pi,\lambda} :\calH_\pi^2 \to \calH_2 (\Xi_H)$ is an
		$S^o$-intertwining operator.
		\item The contribution of the unitary highest weight representations to the Plancherel decomposition of
		$L^2(G/H)$ is given by
		\[ \bigoplus_{\omega(\pi)>c_{G/H}} \beta (\Pi_{H,\pi,\lambda}(\calH_\pi^2)), \]
		where we choose for every $\pi$ a non-trivial distribution vector $\lambda\in(\calH_\pi^2)^{-\infty,H}$ if possible, and otherwise use $\lambda=0$.
		\item  Let $(\pi,\calV_\pi)$ be an irreducible unitary representation of $K$ such that $(\calV_\pi^\vee)^{K\cap H}\not=\{0\}$. Let
		$\eta\in(\calV_\pi^\vee)^{K\cap H}\setminus \{0\}$. Then
		for every $\xi\in\calV_\pi$ the function
		\[ T_{H,\pi,\eta}\xi(x) = \eta\big(\pi(k_H(x)^{-1})\xi\big)\]
is holomorphic on $P^+K_\CC H_\CC /H_\CC$. If
		$\omega (\pi) > c_{G/H}$ then $\{T_{H,\pi,\eta}\xi:\xi\in\calV_\pi\}\subseteq\calH_2 (\Xi_H)$ is the highest weight space of a unitary highest weight representation $\calH_\pi^2$ in $\calH_2(\Xi_H)$.
	\end{enumerate}
\end{theorem}

\begin{remark}
	\begin{enumerate}
		\item We would like to underline the difference that here one needs to use a distribution vector to construct the embedding of the highest weight representation $\calH_\pi^2$ into $L^2(G/H)$. This is a consequence of the fact that there is no $H$-invariant vector in $\calH_\pi^2$. This does not show up in the groups case until one consider the group as a homogeneous space $(G\times G)/\diag(G)$ and identifies $\calH_\pi^2\otimes(\calH_\pi^2)^*$ with the space $\HS(\calH_\pi^2)$ of Hilbert--Schmidt operators on $\calH_\pi^2$. The trace is then the (up to scalar multiples) unique $G$-invariant distribution vector and the intertwining operator is
		\[ T\mapsto\big[s\mapsto\dim(\pi) \tr (T\pi (s^{-1}))\big], \quad T\in\HS(\calH_\pi^2), s\in \Xi, \]
		where $\dim(\pi) $ is the formal dimension of $\calH_\pi^2$.
		\item Part (4) is \cite[Lem 5.1]{O00}, where a Harish-Chandra type formula for $\Lambda_\pi$ as a hyper-function is given. In fact, up to a known constant, $\Lambda_\pi$ is an $H$-invariant spherical functions on the Cartan dual of $G/H$.
		\item If $G/K\simeq \RR^k +i \Omega$ is of tube type, then there exists a symmetric group $H\subseteq G$ locally isomorphic to the automorphism group of the cone $\Omega$. In this case, assuming that $G$ is simple, the group $H$ has one-dimensional center whose connected component $Z(H)_o$ is  a vector group. Each unitary character $\chi : Z(H)_o \to \TT$ defines a line bundle over $G/H$. One can then also consider $\chi$-twisted Hardy spaces over $\Xi_H$.
		\item In general, we might have $c_{G/H} < \frac{2n}{r}-1 $ in which case there are unitary highest weight modules that can be realized inside $L^2(G/H)$ but do not belong to the holomorphic discrete series of $G$. A detailed discussion of this for $G$ classical and $\dim \calV_\pi =1$ can be found in the last part of \cite{OO88}.
	\end{enumerate}
\end{remark}


\begin{thebibliography}{10}
	
	\bibitem{AG21}
	Avraham Aizenbud and Dmitry Gourevitch, \emph{Finite multiplicities beyond
		spherical pairs},  (2021), preprint, available at
	\href{https://arxiv.org/abs/2109.00204}{arXiv:2109.00204}.
	
	\bibitem{BK17}
	Yves Benoist and Toshiyuki Kobayashi, \emph{Tempered homogeneous spaces {II}}, Dynamics, geometry, number theory -- the impact of Margulis on modern mathematics, University of Chicago Press, Chicago, IL, 2022, pp. 213--245.
	
	\bibitem{BM19}
	Roelof W. Bruggeman and Roberto J. Miatello, \emph{Representations of $\SU(2,1)$ in Fourier Term Modules}, Springer, Cham, 2023, Lecture Notes in Math., Vol. 2340.
	
	\bibitem{BM21}
	\bysame, \emph{Generalized Poincar\'{e}
		series for $\SU(2,1)$}, (2021), preprint, available at
	\href{https://arxiv.org/abs/2106.14200}{arXiv:2106.14200}.
	
	\bibitem{BFK18}
	Jan Bruinier, Jens Funke, and Stephen Kudla, \emph{Degenerate {W}hittaker
		functions for {${\rm Sp}_n({\mathbb{R}})$}}, Int. Math. Res. Not. IMRN (2018),
	no.~1, 1--56.
	
	\bibitem{Car76}
	Pierre Cartier, \emph{Vecteurs diff\'{e}rentiables dans les repr\'{e}sentations
		unitaires des groupes de {L}ie}, S\'{e}minaire {B}ourbakt (1974/1975), {E}xp.
	{N}o. 454, 1976, pp.~20--34. Lecture Notes in Math., Vol. 514.
	
	\bibitem{CF04}
	Houcine Ch\'{e}bli and Jacques Faraut, \emph{Fonctions holomorphes \`a
		croissance mod\'{e}r\'{e}e et vecteurs distributions}, Math. Z. \textbf{248}
	(2004), no.~3, 540--565.
	
	\bibitem{Dib90}
	Hacen Dib, \emph{Fonctions de {B}essel sur une alg\`ebre de {J}ordan}, J. Math.
	Pures Appl. (9) \textbf{69} (1990), no.~4, 403--448.
	
	\bibitem{DG93}
	Hongming Ding and Kenneth~I. Gross, \emph{Operator-valued {B}essel functions on
		{J}ordan algebras}, J. Reine Angew. Math. \textbf{435} (1993), 157--196.
	
	\bibitem{EHW83}
      Thomas Enright, Roger Howe, and Nolan Wallach, \emph{A classification of unitary
      highest weight modules}, Representation theory of reductive groups (Park City, Utah, 1982),
      Progr. Math. \textbf{40}, Birkhäuser Boston, Boston, MA, 1983, pp. 97--143. 
	
	\bibitem{FK94}
	Jacques Faraut and Adam Kor\'{a}nyi, \emph{Analysis on symmetric cones}, Oxford
	Mathematical Monographs, The Clarendon Press, Oxford University Press, New
	York, 1994, Oxford Science Publications.
	
	\bibitem{FOO22} Jan Frahm, Gestur \'Olafsson, and Bent {\O}rsted, \emph{Generalized Laguerre functions and Whittaker vectors for holomorphic discrete series}, J. Lie Theory \textbf{33} (2023), no.~1, 253--270.
	
	\bibitem{FOO24}
	\bysame, \emph{The holomorphic discrete series contribution to the generalized Whittaker Plancherel formula II. Non-tube type groups}, (2024), preprint, available at
	\href{https://arxiv.org/abs/2401.06427}{arXiv:2401.06427}.
	
	\bibitem{GG77}
	Israel~M. Gelfand and Simon~G. Gindikin, \emph{Complex manifolds whose spanning
		trees are real semisimple {L}ie groups, and analytic discrete series of
		representations}, Funkcional. Anal. i Prilo\v{z}en. \textbf{11} (1977),
	no.~4, 19--27, 96.
	
	\bibitem{GKO04}
	Simon Gindikin, Bernhard Kr\"{o}tz, and Gestur \'{O}lafsson, \emph{Holomorphic
		{$H$}-spherical distribution vectors in principal series representations},
	Invent. Math. \textbf{158} (2004), no.~3, 643--682.
	
	\bibitem{GKO06}
	\bysame, \emph{Horospherical model for holomorphic discrete series and
		horospherical {C}auchy transform}, Compos. Math. \textbf{142} (2006), no.~4,
	983--1008.
	
	\bibitem{Her55}
	Carl~S. Herz, \emph{Bessel functions of matrix argument}, Ann. of Math. (2)
	\textbf{61} (1955), 474--523.
	
	\bibitem{HOO91}
	Joachim Hilgert, Gestur \'{O}lafsson, and Bent {\O}rsted, \emph{Hardy spaces on
		affine symmetric spaces}, J. Reine Angew. Math. \textbf{415} (1991),
	189--218.
	
	\bibitem{Jak83}
	Hans Plesner Jakobsen, \emph{Hermitian symmetric spaces and their unitary highest
	weight modules}, J. Funct. Anal. \textbf{52} (1983), no.~3, 385--412.
	
	\bibitem{Kna86}
	Anthony~W. Knapp, \emph{Representation theory of semisimple groups: An overview
		based on examples}, Princeton Mathematical Series, vol.~36, Princeton
	University Press, Princeton, NJ, 1986.
	
	\bibitem{Kos00}
	Bertram Kostant, \emph{On {L}aguerre polynomials, {B}essel functions, {H}ankel
		transform and a series in the unitary dual of the simply-connected covering
		group of {${\rm Sl}(2,{\bf R})$}}, Represent. Theory \textbf{4} (2000),
	181--224.
	
	\bibitem{KNO97}
	Bernhard Kr\"{o}tz, Karl-Hermann Neeb, and Gestur \'{O}lafsson, \emph{Spherical
		representations and mixed symmetric spaces}, Represent. Theory \textbf{1}
	(1997), 424--461.
	
	\bibitem{Moe13}
	Jan M\"{o}llers, \emph{A geometric quantization of the {K}ostant-{S}ekiguchi
		correspondence for scalar type unitary highest weight representations}, Doc.
	Math. \textbf{18} (2013), 785--855.
	
	\bibitem{N00}
	Karl-Hermann Neeb, \emph{Holomorphy and convexity in {L}ie theory}, De Gruyter
	Expositions in Mathematics, vol.~28, Walter de Gruyter \& Co., Berlin, 2000.
	
	\bibitem{O00}
	Gestur \'{O}lafsson, \emph{Analytic continuation in representation theory and
		harmonic analysis}, Global analysis and harmonic analysis
	({M}arseille-{L}uminy, 1999), S\'{e}min. Congr., vol.~4, Soc. Math. France,
	Paris, 2000, pp.~201--233.
	
	\bibitem{OO88}
	Gestur \'{O}lafsson and Bent {\O}rsted, \emph{The holomorphic discrete series
		for affine symmetric spaces. {I}}, J. Funct. Anal. \textbf{81} (1988), no.~1,
	126--159.
	
	\bibitem{OO91}
	\bysame, \emph{The holomorphic discrete series of an affine symmetric space and
		representations with reproducing kernels}, Trans. Amer. Math. Soc.
	\textbf{326} (1991), no.~1, 385--405.
	
	\bibitem{O91}
	Grigori~I. Ol'shanski\u{\i}, \emph{Complex {L}ie semigroups, {H}ardy spaces and the
		{G}elfand--{G}indikin program}, Differential Geom. Appl. \textbf{1} (1991),
	no.~3, 235--246.
	
	\bibitem{Shi94}
	Nobukazu Shimeno, \emph{The {P}lancherel formula for spherical functions with a
		one-dimensional {$K$}-type on a simply connected simple {L}ie group of {H}ermitian
		type}, J. Funct. Anal. \textbf{121} (1994), no.~2, 330--388.
	
	\bibitem{S83}
	Robert~J. Stanton, \emph{Analytic extension of the holomorphic discrete
		series}, Amer. J. Math. \textbf{108} (1986), no.~6, 1411--1424.
	
	\bibitem{Wal03}
	Nolan~R. Wallach, \emph{Generalized {W}hittaker vectors for holomorphic and
		quaternionic representations}, Comment. Math. Helv. \textbf{78} (2003),
	no.~2, 266--307.
	
     \bibitem{Y88}
      Hiroshi~Yamashita, \emph{Multiplicity One Theorems for Generalized Gelfand-Graev
      Representations of Semisimple Lie Groups and
      Whittaker Models for the Discrete Series}, Advanced Studies in Pure Mathematics
      \textbf{14}, 1988, Representations of Lie Groups, Kyoto, Hiroshima, 1986,
      pp. 31--121. 
 
\end{thebibliography}

\providecommand{\bysame}{\leavevmode\hbox to3em{\hrulefill}\thinspace}
\providecommand{\href}[2]{#2}

\end{document}